\documentclass[a4paper,10pt]{amsart}
\usepackage{amsmath}
\usepackage{amssymb}
\usepackage{latexsym}
\usepackage{amsbsy}
\usepackage[all]{xy}
\usepackage{amscd}
\usepackage[dvips]{graphicx}

\newtheorem{Prop}{Proposition}[section]
\newtheorem{Thm}[Prop]{Theorem}
\newtheorem{Lemma}[Prop]{Lemma}
\newtheorem{Cor}[Prop]{Corollary}

\newtheorem{Remark}[Prop]{Remark}

\newtheorem{Definition}[Prop]{Definition}

\newcommand{\str}[1]{\langle #1\rangle}
\def\az{\alpha}      \def\ud{{\underline{d}}}

\def\bbn{{\mathbb N}}  \def\bbz{{\mathbb Z}}  \def\bbq{{\mathbb Q}} \def\bb1{{\mathbb 1}}

   \def\bbe{{\mathbb E}} 

  \def\bbc{{\mathbb C}}

\def\ra{\rightarrow}
\def\hom{\mbox{Hom}}

\def\ext{\mbox{Ext}\,} 
\def\dim{\mbox{dim}\,}

\def\uq2{U_q(\hat{sl}_2)}

\def\bb{{\bf b}}

\def\nd{{\noindent}}

\def\mc{{\mathcal{C}}}

\def\md{{\mathcal{D}}}

\def\mo{{\mathcal{O}}}

\def\ue{{\underline{e}}}

\begin{document}

\title[Integral bases of cluster algebras and
tame quivers] {Integral bases of cluster algebras and
representations of tame quivers}

\thanks{The research was partially supported  by NSF of China (No. 10631010) and
the Ph.D. Programs Foundation of Ministry of Education of China (No.
200800030058)}

\author{Ming Ding, Jie Xiao and Fan Xu}
\address{Department of Mathematical Sciences\\
Tsinghua University\\
Beijing 10084, P.~R.~China} \email{m-ding04@mails.tsinghua.edu.cn
(M.Ding),\ jxiao@math.tsinghua.edu.cn (J.Xiao),\
fanxu@mail.tsinghua.edu.cn (F.Xu)}

\maketitle


\bigskip

\begin{abstract}
In \cite{CK2005} and  \cite{SZ}, the authors constructed the bases
of cluster algebras of finite types and of type
$\widetilde{A}_{1,1}$, respectively. In this paper, we will deduce
$\mathbb{Z}$-bases for cluster algebras of affine types.
\end{abstract}

\setcounter{tocdepth}{1}
\tableofcontents

\section{Introduction}
Cluster algebras were introduced by S. Fomin and A. Zelevinsky
\cite{FZ} in order to develop a combinatorial approach to study
problems of total positivity in algebraic groups and canonical
bases in quantum groups. The link between acyclic cluster algebras
and representation theory of quivers were first revealed in
\cite{MRZ}. In~\cite{BMRRT}, the authors introduced the cluster
categories as the categorification of acyclic cluster algebras.
Let $Q$ be an acyclic quiver with vertex set $Q_0=\{1,2,\cdots,
n\}$. Let $A=\bbc Q$ be the path algebra of $Q$ and we denote by
$P_i$ the indecomposable projective $\bbc Q$-module with the
simple top $S_i$ corresponding to $i\in Q_0$ and $I_i$ the
indecomposable injective $\bbc Q$-module with the simple socle
$S_i$. Let $\md^b(Q)$ be the bounded derived category of
$\mathrm{mod} \bbc Q$ with the shift functor $[1]$ and the
AR-translation $\tau$. The cluster category associated to $Q$ is
the orbit category $\mathcal{C}(Q):=\md^{b}(Q)/F$ with
$F=[1]\circ\tau^{-1}$. Let $\mathbb{Q}(x_1,\cdots,x_n)$ be a
transcendental extension of $\mathbb{Q}$ in the indeterminates
$x_1, \cdots, x_n.$ The Caldero-Chapton map of an acyclic quiver
$Q$ is the map
$$X_?^Q: \mathrm{obj}(\mc(Q))\ra\bbq(x_1,\cdots,x_n)$$ defined in \cite{CC} by
the following rules:
            \begin{enumerate}
                \item if $M$ is an indecomposable $\bbc Q$-module, then
                    $$
                        X_M^Q = \sum_{\textbf e} \chi(\mathrm{Gr}_{\ue}(M)) \prod_{i \in Q_0} x_i^{-\left<\ue, s_i\right>-\left <s_i, \underline{\mathrm{dim}}M - \ue\right
                        >};
                    $$
                \item if $M=P_i[1]$ is the shift of the projective module associated to $i \in Q_0$, then $$X_M^Q=x_i;$$
                \item for any two objects $M,N$ of $\mathcal C_Q$,
                we have
                    $$X_{M \oplus N}^Q=X_M^QX_N^Q.$$
            \end{enumerate}
Here, we denote by $\left <-,-\right >$ the Euler form on $\bbc
Q$-mod and $Gr_{\ue}(M)$ is the $\ue$-Grassmannian of $M,$ i.e. the
variety of submodules of $M$ with dimension vector $\ue.$
$\chi(\mathrm{Gr}_{\ue}(M))$ denote its Euler-Poincar$\acute{e}$
characteristic. For any object $M\in \mc(Q)$, $X_M^Q$ will be called
the generalized cluster variable for $M$.

We note that the indecomposable $\bbc Q$-modules and $P_i[1]$ for
$i\in Q_0$ exhaust the indecomposable objects of the cluster
category $\mc(Q)$:
$$\mathrm{ind}\mc(Q)=\mathrm{ind}\bbc Q\sqcup\{P_i[1]:i\in
Q_0\}.$$ Each object $M$ in $\mc(Q)$ can be uniquely decomposed in
the following way:
$$M=M_0\oplus P_M[1]$$
where $M_0$ is a module and $P_M$ is a projective  module. Let
$P_M=\bigoplus_{i\in{Q_0}}m_iP_i.$ We extend the definition of the
dimension vector $\mathrm{\underline{dim}}$ on modules in
$\mathrm{mod}\bbc Q$ to objects in $\mc(Q)$ by setting
$$\mathrm{\underline{dim}}M=\mathrm{\underline{dim}}M_0-(m_i)_{i\in{Q_0}}.$$
 Let $R=(r_{ij})$ be a
matrix of size $|Q_0|\times |Q_0|$ satisfying
$$r_{ij}=\dim_{\bbc}\mathrm{Ext}^1(S_i,S_j)$$ for any $i, j\in Q_0$. The Caldero-Chapton map can be reformulated by the following rules (see \cite{Xu} or \cite{Hubery2005}):

\begin{enumerate}
\item
$$
X_{\tau P}=X_{P[1]}=x^{\underline{\mathrm{dim}}{P/rad P}},
X_{\tau^{-1}I}=X_{I[-1]}=x^{\underline{\mathrm{dim}} socI}$$ for
any projective $\bbc Q$-module $P$ and any injective $\bbc
Q$-module $I$; \item
$$ X_{M}=\sum_{\ue}\chi(\mathrm{Gr}_{\ue}(M))x^{\ue R+(\underline{\mathrm{dim}}M-\ue)R^{tr}-
\underline{\mathrm{dim}}M }
$$
where $M$ is a $\bbc Q$-module, $R^{tr}$ is the transpose of the
matrix $R$ and $x^{v}=x^{v_1}_1\cdots x^{v_n}_n$ for
$v=(v_1,\cdots,v_n)\in \mathbb{Z}^{n}$.
\end{enumerate}

Let $\mathcal{AH}(Q)$ be the subalgebra of $\bbq(x_1,\cdots, x_n)$
generated by
$$\{X_M, X_{\tau P}\mid M, P\in\mathrm{ mod}\bbc Q, P
\mbox{ is projective module}\}$$ and  $\mathcal{EH}(Q)$ be the
subalgebra of $\mathcal{AH}(Q)$ generated by
 $$\{X_{M}\mid M\in \mathrm{ind}\mc(Q), \ext^1_{\mc{(Q)}}(M,M)=0\
 \}.$$

In \cite{CK2006}, the authors showed that the Caldero-Chapton map
 induces a one to one correspondence between indecomposable
 objects  in $\mc(Q)$ without self-extension and the cluster
 variables of the cluster algebra $\mathcal{A}(Q).$
If $Q$ is a simply laced Dynkin quiver , $\mathcal{AH}(Q)$ coincides
with $\mathcal{EH}(Q)$. In \cite{CK2005}, the authors showed that
the set $$\{X_M\mid M\in \mc(Q), \ext^1_{\mc{(Q)}}(M,M)=0 \}$$ is a
$\bbz$-basis of $\mathcal{EH}(Q)$, i.e. the cluster monomials is a
$\bbz$-basis of $\mathcal{EH}(Q)$. If $Q$ is a quiver of
$\widetilde{A}_{1,1}$, $\mathcal{AH}(Q)$ also coincides with
$\mathcal{EH}(Q)$ (see \cite{SZ}, \cite{CZ}). Furthermore, in
\cite{CZ}, the authors gave a $\bbz$-basis of $\mathcal{AH}(Q)$
called the semicanonical basis. If $Q$ is a quiver of type
$\widetilde{D}_4$, $\mathcal{AH}(Q)$ is still equal to
$\mathcal{EH}(Q)$ and a $\bbz$-basis is given in \cite{DX1} (see
also Section \ref{d4}). Recently, in \cite{Dupont1}, Dupont
introduced generic variables for any acyclic cluster algebra and
conjectured that generic variables constitute a $\bbz$-basis for any
acyclic cluster algebra. He proved that the conjecture is true for a
cluster algebra of type $\widetilde{A}_{p,q}$ and for any affine
type, deduce the conjecture to a certain \emph{difference property}.
The difference property has been confirmed for type
$\widetilde{A}_{p,q}$ in \cite{Dupont1}. At the same time, we
independently constructed various $\bbz$-bases for cluster algebras
of tame quivers with alternating orientations (i.e. any vertex is a
sink or a source) in the first version of this paper (see
\cite{DXX}). The present paper is the strengthen version of
\cite{DXX}. It is self-contained. One need not refer to \cite{DXX}.

The main goal of the present paper is to construct various
$\bbz$-bases for the cluster algebra of a tame quiver $Q$ with any
acyclic orientation. Let us describe it in detail. Let $Q$ be a tame
quiver. Then the underlying graph of $Q$ is of affine type
$\widetilde{A}_{p,q}, \widetilde{D}_n (n\geq 4)$ or $\widetilde{E}_m
(m=6,7,8)$ and $Q$ contains no oriented cycles. There are many
references about the theory of representations of tame quivers, for
example, see \cite{DR76} or \cite{CB}.  Define the set
$$
\textbf{D}(Q)=\{\ud\in \bbn^{Q_0}\mid  \exists \mbox{ a regular
module }T\oplus R \mbox{ such that }
\mathrm{\underline{dim}}(T\oplus R)=\ud,$$$$ T \mbox{
indecomposable, } \mathrm{Ext}^1_{\mc(Q)}(T,T)\neq 0,
\mathrm{Ext}^1_{\mc(Q)}(T, R)=\mathrm{Ext}^1_{\mc(Q)}(R,R)=0\}.
$$
We make an assignment, i.e., a map $$\phi: \textbf{D}(Q)\rightarrow
\mathrm{obj}(\mc(Q))$$ and set $$X^{\phi}_{\ud}:=X_{\phi(T\oplus
R)}.$$ It is clear that the above assignment is not unique. For
simplicity and without confusion, we omit $\phi$ in the notation
$X^{\phi}_{\ud}$. The main theorem of this paper is as follows.
\begin{Thm}\label{30}
Let $Q$ be a tame quiver and fix an assignment. Then the set
$$\mathcal{B}(Q):=\{X_{L},X_{\ud}|L\in \mc(Q), \ud\in \mathrm{\mathrm{\textbf{D}}}(Q), \mathrm{Ext}^{1}_{\mc(Q)}(L,L)=0\}$$
is a $\mathbb{Z}$-basis of $\mathcal{EH}(Q).$
\end{Thm}
 Given an assignment, we obtain a
$\bbz$-basis. In fact, by Theorem \ref{30}, we will see that
$\bbz^{Q_0}=\textbf{D}(Q)\cup \textbf{E}(Q)$ where
$\textbf{E}(Q)=\{\ud\in \bbz^{Q_0}\mid \exists L\in \mc(Q)\mbox{
satisfies } \mathrm{\underline{dim}}L=\ud \mbox{ and
}\mathrm{Ext}_{\mc(Q)}^1(L,L)=0\}$ (Corollary \ref{dimension}). Let
$\ud=n\delta+\ud_0$ be the canonical decomposition of $\ud$ (see
\cite{Kac}) and $E$ be an indecomposable regular simple module of
dimension vector $\delta$. We set $T:=E[n]$ to be the indecomposable
regular module with quasi-socle $E$ and quasi-length $n$. Then we
obtain another $\bbz$-basis
$$\mathcal{B}'(Q):=\{X_{L},X_{E[n]\oplus R}|L\in \mc(Q), \mathrm{Ext}^{1}_{\mc(Q)}(L,L)=0, \mathrm{Ext}^{1}_{\mc(Q)}(R,R)=0\}.$$

The paper is organized as follows. In Section 2, we recall various
cluster multiplication formulas. In Section 3 and Section 4, we
characterize the generalize cluster variables $X_M$ for $M\in
\mc(Q)$ and  a quiver $Q$ with an alternating orientation. In
particular, we compare the generalized cluster variables of modules
in homogeneous tubes and non-homogeneous tubes (see Proposition
\ref{5}). Note that the proof of Proposition \ref{5} does not depend
on the difference property unless $Q$ is of type  $\widetilde{A}$.
In Section 5, we prove Theorem \ref{30} by using the BGP-reflection
functor between cluster categories defined in \cite{Zhu} and the
construction of $\bbz$-bases for cluster algebras of type
$\widetilde{A}_{p, q}$ where $p\neq q$ in \cite{Dupont1}. The
BGP-reflection functor between cluster categories admits that one
can drop the assumption of alternating orientation. We illustrate
Theorem \ref{30} by two examples of type $\widetilde{A}_{1,1}$ and
$\widetilde{D}_4$ in Section 6. Section 7 shows the inductive
formulas for the multiplication between generalized cluster
variables of modules in a tube.  In order to compare our bases with
generic variables in \cite{Dupont1}, we prove the difference
property for any  tame quiver (Theorem \ref{difference}) in the
appendix of the present paper (Section 8). A direct corollary is
that the set of generic variables (denoted by $\mathcal{B}_g(Q)$)
defined in \cite{Dupont1} is a $\bbz$-basis of a cluster algebra for
affine type. There is a unipotent transition matrix from
$\mathcal{B}_g(Q)$ to $\mathcal{B}'(Q)$ (Corollary \ref{un}).

After the present paper appear, the latest advance is that
Geiss-Leclerc-Schr\"oer proved the above Dupont's conjecture (in the
more general context) via the construction of the dual semicanonical
bases for preprojective algebras (arXiv:1004.2781).

\section{The cluster multiplication formulas}
In this section, we recall various cluster multiplication
formulas. In Section 7, we will show inductive cluster
multiplication formulas for a tube.
\begin{subsection}{} We recall the cluster multiplication theorem in
\cite{XiaoXu} and \cite{Xu}. It is a generalization of the cluster
multiplication theorem for finite type \cite{CK2005} and for
affine type \cite{Hubery2005}\cite{Hubery2007}. First, we
introduce some notations in \cite{XiaoXu} and \cite{Xu}. Let
$Q=(Q_0,Q_1,s,t)$ be a finite acyclic quiver, where $Q_0$ and
$Q_1$ are the finite sets of vertices and arrows, respectively,
and $s,t: Q_1\rightarrow Q_0$ are maps such that any arrow $\az$
starts at $s(\az)$ and terminates at $t(\az).$ For any dimension
vector
 $\ud=(d_i)_{i\in Q_0},$ we consider the affine space over $\bbc$
$$\bbe_{\ud}=\bbe_{\ud}(Q)=\bigoplus_{\az\in Q_1}\hom_{\bbc}(\bbc^{d_{s(\az)}},\bbc^{d_{t(\az)}}).$$
Any element $x=(x_{\az})_{\az\in Q_1}$ in $\bbe_{\ud}$ defines a
representation $M(x)=(\bbc^{{\ud}},x)$ where
$\bbc^{{\ud}}=\bigoplus_{i\in Q_0}\bbc^{d_i}$. We set
$(M(x))_i:=\bbc^{d_i}.$ Naturally we can define the action of the
algebraic group $G_{{\ud}}(Q)=\prod_{i\in Q_0}GL(\bbc^{d_i})$ on
$\bbe_{\ud}$ by
$g.x=(g_{t(\alpha)}x_{\alpha}g^{-1}_{s(\alpha)})_{\alpha\in Q_1}.$
Let $\ue$ be the dimension vector
$\mathrm{\underline{dim}}M+\mathrm{\underline{dim}}N$. For $L\in
\bbe_{\ue},$ we define
$$\str{L}:=\{L'\in \bbe_{\ue}\mid
\chi(\mathrm{Gr}_{\ud}(L'))=\chi(\mathrm{Gr}_{\ud}(L)) \mbox{ for
any }\ud\}.$$  There exists a finite subset $S(\ue)$ of $\bbe_{\ue}$
such that (\cite[Corollary 1.4]{Xu})
$$\bbe_{\ue}=\bigsqcup_{L\in S(\ue)}\str{L}.$$ Let $\mo$ is a
$G_{\ue}$-invariant constructible subset of $\bbe_{\ue}$. Define
$$
\mathrm{Ext}^1(M,N)_{\mo}=\{[0\rightarrow N\rightarrow
L\rightarrow M\rightarrow 0]\in \mathrm{Ext}^1(M,N)\setminus\{0\}
\mid L\in \mo\}
$$
There is a $\bbc^*$-action on $\mathrm{Ext}^1(M,N)$. The orbit
space is denoted by $\mathbb{P}\mathrm{Ext}^1(M,N).$ Given $\bbc
Q$-modules $U, V$ and an injective module $J$, define
$\mathrm{Hom}_{\bbc Q}(N, \tau M)_{\str{V}\oplus \str{U}\oplus
J[-1]}$ is the set
$$
\{f\in \mathrm{Hom}_{\bbc Q}(N, \tau M)\mid \mathrm{Ker}(f)\in
\str{V}, \mathrm{Coker}(f)=\tau U'\oplus J \mbox{ for some } U'\in
\str{U}\}
$$
\end{subsection}
For any projective $\bbc Q$-module $P$, let $I=\mathrm{DHom}_{\bbc
Q}(P,\bbc Q).$  Then there exist two finite partitions
$$
\hom(M, I)=\bigsqcup_{I', V\in
S(\ud_1(I'))}\hom(M,I)_{\str{V}\oplus I'[-1]},
$$
$$
\hom(P, M)=\bigsqcup_{P', U\in
S(\ud_2(P'))}\hom(P,M)_{P'[1]\oplus\str{U}},
$$
where
$\ud_1(I')=\mathrm{\underline{dim}}I+\mathrm{\underline{dim}}I'-\mathrm{\underline{dim}}M$,
$\ud_2(P')=\mathrm{\underline{dim}}P+\mathrm{\underline{dim}}P'-\mathrm{\underline{dim}}M$,
$$\hom(M,I)_{\str{V}\oplus I'[-1]}=\{f\in \hom(M, I)\mid
\mathrm{Ker}(f)\in \str{V}, \mathrm{Coker}(f)=I' \} $$ and $$
\hom(P,M)_{P'[1]\oplus\str{U}}=\{g\in \hom(P,M)\mid
\mathrm{Ker}(g)=P', \mathrm{Coker}(g)\in \str{U}\}. $$

\begin{Thm}\cite{XiaoXu}\cite{Xu}\label{XX}
Let $Q$ be an acyclic quiver. Then

\nd (1) for any $\bbc Q$-modules $M,N$ such that $M$ contains no
projective summand, we have
$$\hspace{0cm}\mathrm{dim}_{\bbc}\mathrm{Ext}^1_{\bbc Q}(M,N)X_{M} X_{N} =\sum_{L\in S(\ue)}\chi(\mathbb{P}\mathrm{Ext}^{1}_{\bbc Q}(M,N)_{\str{L}})X_{L}$$
$$
+\sum_{I, \ud_1,\ud_2}\sum_{ V\in S(\ud_1),U\in
S(\ud_2)}\chi(\mathbb{P}\mathrm{Hom}_{\bbc Q}(N,\tau
M)_{\str{V}\oplus \str{U}\oplus I[-1]})X_{V\oplus U\oplus I[-1]}
$$
where $\ue=\mathrm{\underline{dim}}M+\mathrm{\underline{dim}} N$
and

\nd (2) for any $\bbc Q$-module $M$ and projective $\bbc Q$-module
$P$, we have
 $$\mathrm{dim}_{\bbc}\mathrm{Hom}_{\bbc Q}(P,M)X_{M}X_{P[1]}=\sum_{I',V\in S(\ud_1(I'))}
 \chi(\mathbb{P}\mathrm{Hom}_{\bbc Q}(M,I)_{\str{V}\oplus I'[-1]
})X_{V\oplus I'[-1]}$$$$\hspace{4.3cm}+\sum_{P', U\in
S(\ud_2(P'))}\chi(\mathbb{P}\mathrm{Hom}_{\bbc
Q}(P,M)_{P'[1]\oplus \str{U}})X_{P'[1]\oplus U}
 $$
where  $I=\mathrm{DHom}_{\bbc Q}(P,\bbc Q),$ $I'$ is injective,
and $P'$ is projective.
\end{Thm}

We have the following simplified version of Theorem \ref{XX}.
\begin{Thm}\cite{Palu}\cite{DX2}\label{unified}
Let $kQ$ be an acyclic quiver and $\mc (Q)$ be the cluster category
of $kQ$. Then for $M, N\in \mc(Q)$, there exists a finite subset
$\mathcal{Y}$ of the middle term of extensions in $Ext^1_{\mc(Q)}(M,
N)$ such that  if $\mathrm{Ext}^1_{\mc}(M, N)\neq 0$, we have
$$
\chi(\mathbb{P}\mathrm{Ext}^1_{\mc}(M, N))X_MX_N=\sum_{Y\in
\mathcal{Y}}(\chi(\mathbb{P}\mathrm{Ext}^1_{\mc(Q)}(M, N)_{\langle
Y\rangle})+\chi(\mathbb{P}\mathrm{Ext}^1_{\mc(Q)}(N, M)_{\langle
Y\rangle}))X_Y.
$$
\end{Thm}

\begin{subsection}{}
Let us illustrate the cluster multiplication theorem by the
following example. Let $Q$ be a tame quiver with minimal imaginary
root $\delta$. Assume there is a sink $e\in Q_0$ such that
$\delta_e=1.$ Let $P_e$ be simple projective $\bbc Q$-module at $e.$
There exists unique preinjective module $I$ with dimension vector
$\delta-\mathrm{\underline{dim}}P_e.$ Then we have
$\mathrm{dim}_{\bbc}\mathrm{Ext}^1(I, P_e)=2.$ The set of
indecomposable regular $\bbc Q$-modules consists of tubes indexed by
the projective line $\mathbb{P}^1.$ Let $\mathcal{T}_1, \cdots,
\mathcal{T}_t$ be all non-homogeneous tube. Note that $t\leq 3$. Any
$\bbc Q$-module $L$ with $\mathrm{Ext}^1(I, P_e)_{\str{L}}\neq 0$
belongs to a tube. Conversely, for any tube $\mathcal{T},$ up to
isomorphism, there is unique $L\in \mathcal{T}$ such that
$\mathrm{Ext}^1(I, P_e)_{\str{L}}\neq 0$.
 Let $\mathcal{T}$ and $\mathcal{T'}$ be two homogeneous tubes and
$E, E'$ be the regular simple modules in $\mathcal{T}$ and
$\mathcal{T'}$, respectively. In Section 4.1, we will prove
$\chi(Gr_{\ue}(E))=\chi(Gr_{\ue}(E'))$ for any $\ue.$ Using this
result, we have $X_{E}=X_{E'}.$ By Theorem \ref{XX}, we obtain
$$
2X_{P_e}X_{I}=(2-t)X_{E}+\sum_{i=1}^tX_{E_i}+\sum_{I', U\in
S(\ud(I'))}\chi(\mathbb{P}\mathrm{Hom}(P_e, \tau I)_{\str{U}\oplus
I'[-1]})X_{U\oplus I'[-1]}
$$
where $E$ is any regular simple module with
$\mathrm{\underline{dim}}E=\delta$ and $E_i$ is the unique regular
module in $\mathcal{T}_i$ such that $\mathrm{Ext}^1(I,
P_e)_{E_i}\neq 0.$ Here $I'$ is an injective $\bbc Q$-module and
$\tau(\ud'(I'))=\mathrm{\underline{dim}}\tau
I-\mathrm{\underline{dim}}P_e-\mathrm{\underline{dim}}I'.$
\end{subsection}

\begin{subsection}{}We give some variants of the cluster multiplication theorem.
In fact, we will mainly meet the cases in the following theorems
when we construct integral bases of the cluster algebras of affine
types.
\begin{Lemma}\cite{BMRRT}
Let $Q$ be an acyclic quiver and $\mc(Q)$ be the cluster category
associated to $Q$. Let $M$ and $N$ be indecomposable $\bbc
Q$-modules. Then
$$
\mathrm{Ext}^1_{\mc(Q)}(M,N)\cong \mathrm{Ext}^1_{\bbc
Q}(M,N)\oplus \mathrm{Hom}_{\bbc Q}(M, \tau N).
$$
\end{Lemma}

\begin{Thm}\cite{CC}\label{AR formula}
Let $Q$ be an acyclic quiver and  $M$ any indecomposable
non-projective $\bbc Q$-module, then
$$
X_{M}X_{\tau M}=1+X_{E}
$$
where $E$ is the middle term of the Auslander-Reiten sequence
ending in $M$.
\end{Thm}

\begin{Thm}\cite{CK2006}\label{CK}
Let $Q$ be an acyclic quiver and  $M$,$N$ be any two objects in
$\mathcal{C}(Q)$ such that $\mathrm{dim}_{\bbc}\
\mathrm{Ext}^{1}_{\mathcal{C}(Q)}(M,N)=1$, then
$$
X_{M}X_{N}=X_{B}+X_{B'}
$$
where $B$ and $B'$ are the unique objects such that there exists
non-split triangles
$$M\longrightarrow B\longrightarrow N\longrightarrow M[1]\ and\ N\longrightarrow B'\longrightarrow M\longrightarrow N[1].$$
\end{Thm}

\end{subsection}
\section{Numerators of Laurent expansions in generalized cluster variables}
Let $Q$ be an acyclic quiver and $E_i[n]$ be the indecomposable
regular module with quasi-socle $E_i$ and quasi-length $n$. For any
$\bbc Q$-module $M$, we denoted by $d_{M,i}$ the $i$-th component of
$\underline{\mathrm{dim}} M$.
\begin{Definition}\label{p}
For $M, N\in \mc(Q)$ with
$\mathrm{\underline{dim}}M=(m_{1},\cdots,m_{n})$ and
$\mathrm{\underline{dim}}N=(r_{1},\cdots,r_{n})$, we write
$\mathrm{\underline{dim}}M\preceq \mathrm{\underline{dim}}N$ if
$m_{i}\leq r_{i}\ for\ 1\leq i\leq n$. Moreover, if there exists
some i such that $m_{i}< r_{i}$, then we write
$\mathrm{\underline{dim}}M\prec \mathrm{\underline{dim}}N.$

\end{Definition}

For any $\ud\in \bbz^{Q_0},$ define $\ud^{+}=(d^+_i)_{i\in Q_0}$
such that $d^+_i=d_i$ if $d_i>0$ and $d_i^+=0$ if $d_i\leq 0$ for
any $i\in Q_0.$ Dually, we set $\ud^-=\ud^+-\ud.$

By Theorem \ref{XX}, we have the following easy lemma.
\begin{Lemma}\label{induction}
With the above notation in Theorem \ref{XX}(1) and
$\mathrm{Ext}^1_{\bbc Q}(M,N)_{\str{L}}\neq 0$, we have $
\mathrm{\underline{dim}}(V\oplus U\oplus I[-1])\prec
\mathrm{\underline{dim}}(M\oplus N)=\mathrm{\underline{dim}}L$.

\end{Lemma}
According to the definition of the Caldero-Chapton map, we consider
the Laurent expansions in generalized cluster variables
$X_{M}=\frac{P(x)}{\Pi_{1\leq i\leq n}x^{m_{i}}_{i}}$ for
$M\in\mathcal{C}(Q)$ such that the integral polynomial $P(x)$ in the
variables $x_i$  is not divisible by any $x_{i}.$ We define the
denominator vector of $X_{M}$ as $(m_{1},\cdots,m_{n})$ \cite{FZ1}.
The following theorem is called as the denominator theorem.
\begin{Thm}\cite{CK2006}\label{d}
Let Q be an acyclic quiver. Then for any object M in $\mc(Q)$, the
denominator vector of $X_{M}$ is $\mathrm{\underline{dim}}M.$
\end{Thm}

The orientation of a quiver $Q$ is called \emph{alternating} if
every vertex of $Q$ is a sink or a source. We note that there exists
an alternating orientation for a quiver of type
$\widetilde{A}_{p,p}, \widetilde{D}_n (n\geq 4)$ or $\widetilde{E}_m
(m=6,7,8)$. According to Theorem \ref{d}, we can prove the following
propositions.
\begin{Lemma}\label{1}
Let $Q$ be a tame quiver with the \emph{alternating} orientation. If
$M$ is either $P_i$ or $I_i$ for $1\leq i\leq n$, then
$X_{M}=\frac{P(x)}{x^{\mathrm{\underline{dim}}(M)}}$ where the
constant term of P(x) is 1.
\end{Lemma}
\begin{proof}
1) If $i$ is a sink point, we have the following short exact
sequence:
$$0\longrightarrow P_i\longrightarrow I_i\longrightarrow I'\longrightarrow 0$$
Then by  Theorem \ref{XX}, we have:
$$X_{\tau P_i}X_{P_i}=x^{\mathrm{\underline{dim}}\mathrm{soc} I'}+1$$
Thus the constant term of numerator in  $X_{P_i}$ as an
irreducible fraction of integral polynomials in the variables
$x_i$ is $1$ because of $X_{\tau P_i}=x_i.$

 If $i$ is a source point, we have the following short exact
sequence:
$$0\longrightarrow P'\longrightarrow P_i\longrightarrow I_i\longrightarrow 0$$
Similarly we have:
$$X_{\tau P_i}X_{P_i}=X_{P'}+1$$
Thus we can finish it by induction on $P'$.

2) For $X_{I_i}$, it is totally similar.
\end{proof}
Note that $X_{\tau {P_i}}=x_i=\frac{1}{x^{-1}_i}$ and
$\mathrm{\underline{dim}}(\tau {P_i})=(0,\cdots,0,-1,0,\cdots,0)$
with i-th component 1 and others 0. Hence we denote the
denominator of $X_{\tau {P_i}}$ by $x^{-1}_i$, and assert the
constant term of numerator in $X_{\tau {P_i}}$ is 1. With these
notations, we have the following Proposition \ref{2}.
\begin{Prop}\label{2}
Let $Q$ be a tame quiver with the \emph{alternating} orientation.
For any object $M \in \mathcal{C}(Q)$, then
$X_{M}=\frac{P(x)}{x^{\mathrm{\underline{dim}}(M)}}$ where the
constant term of P(x) is 1.
\end{Prop}
\begin{proof}
It is enough to consider the case that $M$ is an indecomposable
module.

1) If $M$ is an indecomposable preprojective module, then by the
exchange relation in Thereom \ref{AR formula} we have
$$X_{M}X_{\tau M}=\prod_iX_{B_i}+1.$$
Thus by Lemma \ref{1}, we can prove that
$X_{M}=\frac{P(x)}{x^{\mathrm{\underline{dim}}(M)}}$ where the
constant term of $P(x)$ is $1$ by induction and the directness of
the preprojective component of the Auslander-Reiten quiver of
indecomposable $\bbc Q$-modules. It is similar for indecomposable
preinjective modules.

2) If $M$ is an indecomposable regular module, it is enough to prove
that the proposition holds for any regular simple module according
to the exchange relation. If $M$ is in a homogeneous tube, , then
$M\cong \tau M$. It is enough to consider the case of
$\mathrm{\underline{dim}}M=\delta=(\delta_i)_{i\in Q_0}$ by Theorem
\ref{AR formula}. Note that there exists a vertex $e\in Q_0$ such
that $\delta_e=1.$ Thus we have
$$\mathrm{dim}_{\bbc} \mathrm{Ext}^1_{\bbc
Q}(M,P_e)=\mathrm{dim}_{\bbc}\mathrm{Hom}_{\bbc Q}(P_e, M)=1.$$
Then we obtain the following two non-split exact sequences:
$$0\longrightarrow P_e\longrightarrow L\longrightarrow M\longrightarrow 0$$
and
$$0\longrightarrow L'\longrightarrow P_e\longrightarrow  M\longrightarrow L''\longrightarrow 0$$
where $L$ and $L'$ are preprojective modules and $L''$  is a
preinjective module. Using Theorem \ref{XX} or Theorem \ref{CK},
we have
$$X_{M}X_{P_e}=X_{L}+X_{L'}X_{\tau^{-1} L''}$$
where $\mathrm{\underline{dim}}(L'\oplus \tau^{-1} L'')\prec
\mathrm{\underline{dim}}(P_e\oplus M)$. We have already proved
that the constant terms of the numerators of $X_{P_e}, X_{L},
X_{L'}$ and $X_{\tau^{-1}L''}$ as irreducible fractions of
integral polynomials in the variables $x_i$ is $1$ by the
discussion in 1), then the constant term of the numerator in
$X_{M}$ as an irreducible fraction must be $1$. Let $\mathcal{T}$
be a non-homogeneous tube of rank $r$ with regular simple module
$E_i$ for $1\leq i\leq r.$ By Theorem \ref{AR formula}, we only
need to prove the constant term of the numerator in $X_{E_{i}}$ is
$1$ for $1\leq i\leq r$. We assume $d_{E_r, e}\neq 0$ and $\tau
E_{2}=E_{1}, \cdots, \tau E_{1}=E_r.$ Therefore
$\mathrm{dim}_{\bbc}\mathrm{Ext}^1(E_{1},P_e)=1$, then we have the
following non-split exact sequences combining the relation $\tau
E_{1}=E_{r}$
$$0\longrightarrow P_e\longrightarrow L\longrightarrow E_{1}\longrightarrow 0$$
and
$$0\longrightarrow L'\longrightarrow P_e\longrightarrow E_{r}\longrightarrow L''\longrightarrow 0$$
where $L$ and $L'$ are preprojective modules and $L''$  is a
preinjective module. Then  we have
$$X_{E_{1}}X_{P_e}=X_{L}+X_{L'}X_{\tau^{-1} L''}$$
where $\mathrm{\underline{dim}}(L'\oplus \tau^{-1} L'')\prec
\mathrm{\underline{dim}}(P_e\oplus E_{1})$. Hence, the constant
term of the numerator in $X_{E_{1}}$ must be $1$.  Note that
$d_{E_{r}[2],e}=d_{E_{r},e}+ d_{E_{1},e}=1$, by similar
discussions, we can obtain the constant term of the numerator in
$X_{E_{1}[2]}$ must be $1$. Thus by
$X_{E_{1}}X_{E_{2}}=X_{E_{1}[2]}+1$, we obtain that the constant
term of the numerator in $X_{E_{2}}$ must be $1$. Using the same
method, we can prove the constant term of the numerator in
$X_{E_{i}}$ must be $1$ for $3\leq i\leq r.$
\end{proof}

\section{Generalized
cluster variables on tubes} Let $Q$ be a tame quiver with the
minimal imaginary root $\delta=(\delta_i)_{i\in Q_0}$. Then tubes of
indecomposable regular $\bbc Q$-modules are indexed by the
projective line $\mathbb{P}^1.$ Let $\lambda$ be an index of a
homogeneous tube and $E(\lambda)$ be the regular simple $\bbc
Q$-module with dimension vector $\delta$ in this homogeneous tube.
Assume there are $t(\leq 3)$ non-homogeneous tubes for $Q.$ We
denote these tubes by $\mathcal{T}_1, \cdots, \mathcal{T}_t$. Let
$r_i$ be the rank of $\mathcal{T}_i$ and the regular simple modules
in $\mathcal{T}_i$ be $E^{(i)}_{1}, \cdots E^{(i)}_{r_i}$ such that
$\tau E^{(i)}_2=E^{(i)}_1, \cdots, \tau E^{(i)}_1=E^{(i)}_{r_i}$ for
$i=1, \cdots, t$. If we restrict the discussion to one tube, we will
omit the index $i$ for convenience. Given a regular simple $E$ in a
non-homogeneous tube, $E[i]$ is the indecomposable regular module
with quasi-socle $E$ and quasi-length $i$ for any $i\in \bbn$. Let
$X_M$ be the generalized cluster variable associated to $M$ by the
reformulation of the Caldero-Chapton map. Set
$X_{n\delta_{i,j}}=X_{E^{(i)}_{j}[nr_i]}$ for $n\in \mathbb{N}$.

\begin{Prop}\cite[Lemma 3.14]{Dupont1}\label{independent}
Let $\lambda$ and $\mu$ be in $\bbc$ such that $E(\lambda)$ and
$E(\mu)$ are two regular simple modules of dimension vector
$\delta$. Then $\chi(Gr_{\ue}(E(\lambda)))=\chi(Gr_{\ue}(E(\mu))).$
\end{Prop}

\begin{Prop}\label{def1}
Let $M$ be the regular simple $\bbc Q$-module with dimension vector
$\delta$ in a homogeneous tube. For any $m,n\in \bbn$ and $m\geq n$,
we have
$$X_{M[m]}X_{M[n]}=X_{M[m+n]}+X_{M[m+n-2]}+\cdots+X_{M[m-n+2]}+X_{M[m-n]}.$$
\end{Prop}

\begin{proof}
If $n=1$, we know
$\mathrm{dim}_{\bbc}\mathrm{Ext}^{1}(M[m],M)=\mathrm{dim}_{\bbc}
\mathrm{Hom}(M, M[m])=1$. The involving non-split short exact
sequences are
$$0\longrightarrow M\longrightarrow M[m+1]\longrightarrow M[m]\longrightarrow 0$$
and
$$0\longrightarrow M\longrightarrow M[m]\longrightarrow M[m-1]\longrightarrow 0.$$
Thus by  Theorem \ref{XX} or Theorem \ref{CK}  and the fact $\tau
M[k]=M[k]$ for any $k\in \bbn$, we obtain the equation
 $$
X_{M[m]}X_{M}=X_{M[m+1]}+X_{M[m-1]}.$$

Suppose that it holds for  $n\leq k$. When $n=k+1,$  we have
$$X_{M[m]}X_{M[k+1]}=X_{M[m]}(X_{M[k]}X_{M}-X_{M[k-1]})= X_{M[m]}X_{M[k]}X_{M}-X_{M[m]}X_{M[k-1]}$$
$$\hspace{-1.3cm}=\sum_{i=0}^{k}X_{M[m+k-2i]}X_{M}-\sum_{i=0}^{k-1}X_{M[m+k-1-2i]}$$
$$\hspace{1.3cm}=\sum_{i=0}^{k}(X_{M[m+k+1-2i]}+X_{M[m+k-1-2i]})-\sum_{i=0}^{k-1}X_{M[m+k-1-2i]}$$
$$
\hspace{-4.6cm}=\sum_{i=0}^{k+1}X_{M[m+k+1-2i]}.$$
\end{proof}
 By
Proposition \ref{independent} and Proposition \ref{def1}, we can
define $X_{n\delta}:=X_{M[n]}$ for $n\in \bbn.$ Now we consider
non-homogeneous tubes.

\begin{Prop}\label{4}
Let $Q$ be a tame quiver with the \emph{alternating} orientation.
Then
$X_{n\delta_{i,k}}=X_{n\delta_{j,l}}+\sum_{\mathrm{\underline{dim}}L\prec
n\delta}a_LX_L$  in  non-homogeneous tubes where $a_L\in
\mathbb{Q}$.
\end{Prop}
\begin{proof}
Denote $\delta=(v_1,v_2,\cdots,v_n)$ and
$\mathrm{\underline{dim}}E^{(i)}_{j}=(v_{j1},v_{j2},\cdots,v_{jn})$,
then $\delta=\sum_{1\leq j\leq
r_i}\mathrm{\underline{dim}}E^{(i)}_{j}$. Thus by Theorem
\ref{XX}, Lemma \ref{induction} and the fact that for any
dimension vector there is at most one exceptional module up to
isomorphism, we have
\begin{eqnarray*}
   && X^{v_1}_{S_1}X^{v_2}_{S_2}\cdots X^{v_n}_{S_n}\nonumber \\
   &=& (X^{v_{11}}_{S_1}X^{v_{12}}_{S_2}\cdots X^{v_{1n}}_{S_n})
\cdots (X^{v_{r_i1}}_{S_1}X^{v_{r_i2}}_{S_2}\cdots
X^{v_{r_in}}_{S_n}) \nonumber\\
  &=& (a_1X_{E^{(i)}_{1}}+\sum_{\mathrm{\underline{dim}}L'\prec
\mathrm{\underline{dim}}E^{(i)}_{1}}a_{L'}X_{L'})\cdots
(a_{r_i}X_{E^{(i)}_{r_i}}+\sum_{\mathrm{\underline{dim}}L''\prec
\mathrm{\underline{dim}}E^{(i)}_{r_i}}a_{L''}X_{L''})\nonumber \\
  &=& a_1\cdots a_{r_i}X_{\delta_{i,k}}+\sum_{\mathrm{\underline{dim}}M\prec
\delta}a_MX_M.
\end{eqnarray*}
Note that by Proposition \ref{2}, the left hand side of the above
equation have a term $\frac{1}{\prod^{n}_{i=1}x^{v_i}_{i}}$ which
cannot appear in $X_M$ for $\mathrm{\underline{dim}}M\prec \delta$,
we thus have $a_i's\neq 0$.
 Similarly we have
\begin{eqnarray*}
   && X^{v_1}_{S_1}X^{v_2}_{S_2}\cdots X^{v_n}_{S_n}\nonumber \\
  &=& (b_1X_{E^{(j)}_{1}}+\sum_{\mathrm{\underline{dim}}T'\prec
\mathrm{\underline{dim}}E^{(j)}_{1}}b_{T'}X_{T'})\cdots
(b_{r_j}X_{E^{(j)}_{r_j}}+\sum_{\mathrm{\underline{dim}}T''\prec
\mathrm{\underline{dim}}E^{(j)}_{r_j}}b_{T''}X_{T''})\nonumber \\
   &=& b_1\cdots b_{r_j}X_{\delta_{j,l}}+\sum_{\mathrm{\underline{dim}}N\prec
\delta}b_NX_N.
\end{eqnarray*}
where $b_i's\neq 0$. Thus we have
$$a_1\cdots a_nX_{\delta_{i,k}}=b_1\cdots b_nX_{\delta_{j,l}}+\sum_{\mathrm{\underline{dim}}N\prec
\delta}b_NX_N-\sum_{\mathrm{\underline{dim}}M\prec \delta}a_MX_M.$$
Therefore by Proposition \ref{2} and  Theorem \ref{d}, we have
$$X_{\delta_{i,k}}=X_{\delta_{j,l}}+\sum_{\mathrm{\underline{dim}}N'\prec
\delta}b_{N'}X_{N'}.$$ \\
Now, suppose the proposition holds for $m\leq n,$ then on the one
hand
$$X_{n\delta_{i,k}}X_{\delta_{i,k}}=X_{(n+1)\delta_{i,k}}+\sum_{\mathrm{\underline{dim}}L'\prec
(n+1)\delta}b_{L'}X_{L'}.$$ \\
On the other hand \begin{eqnarray}
  && X_{n\delta_{i,k}}X_{\delta_{i,k}} \nonumber\\
   &=& (X_{n\delta_{j,l}}+\sum_{\mathrm{\underline{dim}}L\prec
n\delta}a_LX_L)X_{\delta_{i,k}} \nonumber\\
   &=& (X_{n\delta_{j,l}}+\sum_{\mathrm{\underline{dim}}L\prec
n\delta}a_LX_L)(X_{\delta_{j,l}}+\sum_{\mathrm{\underline{dim}}N'\prec
\delta}b_{N'}X_{N'})\nonumber \\
   &=& X_{(n+1)\delta_{j,l}}+\sum_{\mathrm{\underline{dim}}L''\prec
(n+1)\delta}b_{L''}X_{L''}.\nonumber
\end{eqnarray}\\
Therefore, we have
$$X_{(n+1)\delta_{i,k}}=X_{(n+1)\delta_{j,l}}+\sum_{\mathrm{\underline{dim}}L''\prec
(n+1)\delta}b_{L''}X_{L''}-\sum_{\mathrm{\underline{dim}}L'\prec
(n+1)\delta}b_{L'}X_{L'}.$$ \\
Thus the proof is finished.
\end{proof}
\begin{Prop}\label{5}
Let $Q$ be a tame quiver with the \emph{alternating} orientation.
Then
$X_{n\delta}=X_{n\delta_{i,1}}+\sum_{\mathrm{\underline{dim}}L\prec
n\delta}a_LX_L$, where $a_L\in \mathbb{Q}$.
\end{Prop}
\begin{proof}
Let $Q$ be of type $\widetilde{D}_n(n\geq 4)$ or
$\widetilde{E}_m(m=6,7,8)$ and $\delta=(v_1,v_2,\cdots,v_n)$,  as
in the proof of  Proposition \ref{4}, we have

\begin{eqnarray*}
   && X^{v_1}_{S_1}X^{v_2}_{S_2}\cdots X^{v_n}_{S_n}\nonumber\\
   &=& (a_1X_{E^{(i)}_{1}}+\sum_{\mathrm{\underline{dim}}L'\prec
\mathrm{\underline{dim}}E^{(i)}_{1}}a_{L'}X_{L'})\cdots
(a_{r_i}X_{E^{(i)}_{r_i}}+\sum_{\mathrm{\underline{dim}}L''\prec
\mathrm{\underline{dim}}E^{(i)}_{r_i}}a_{L''}X_{L''})\nonumber \\
   &=& a_1\cdots a_{r_i}X_{\delta_{i,1}}+\sum_{\mathrm{\underline{dim}}M\prec
\delta}a_MX_M.
\end{eqnarray*}
where $a_i's\neq 0$.
 On the other hand, by the discussion in Section
2.2 and Proposition \ref{4}, we have
\begin{eqnarray*}
   && X^{v_1}_{S_1}X^{v_2}_{S_2}\cdots X^{v_n}_{S_n}\nonumber \\
   &=& (b_1X_{P_{e}}+\sum_{\mathrm{\underline{dim}}L'\prec
\mathrm{\underline{dim}}P_{e}}b_{L'}X_{L'})(b_2X_{I}+\sum_{\mathrm{\underline{dim}}L''\prec
\mathrm{\underline{dim}}I}b_{L''}X_{L''})\nonumber \\
   &=& \frac{1}{2}b_1b_2(-X_{\delta}+\sum^{3}_{k=1}X_{\delta_{k,1}})+\sum_{\mathrm{\underline{dim}}N\prec
\delta}b_{N}X_{N}\nonumber \\
   &=& -\frac{1}{2}b_1b_2X_{\delta}+\frac{3}{2}b_1b_2X_{\delta_{i,1}}+\sum_{\mathrm{\underline{dim}}N'\prec
\delta}b_{N'}X_{N'}.
\end{eqnarray*}
where $b_i's\neq 0$.
 Thus
$$a_1\cdots a_nX_{\delta_{i,1}}+\sum_{\mathrm{\underline{dim}}M\prec
\delta}a_MX_M=-\frac{1}{2}b_1b_2X_{\delta}+\frac{3}{2}b_1b_2X_{\delta_{i,1}}+\sum_{\mathrm{\underline{dim}}N'\prec
\delta}b_{N'}X_{N'}.$$ This deduces the following identity
$$\frac{1}{2}b_1b_2X_{\delta} =(\frac{3}{2}b_1b_2-a_1\cdots a_n)X_{\delta_{i,1}}+\sum_{\mathrm{\underline{dim}}N'\prec
\delta}b_{N'}X_{N'}-\sum_{\mathrm{\underline{dim}}M\prec
\delta}a_MX_M.$$ Note that $b_i's\neq 0$. Therefore by Proposition
\ref{2} and Theorem \ref{d}, we have
$$X_{\delta}=X_{\delta_{i,1}}+\sum_{\mathrm{\underline{dim}}M'\prec
\delta}a_{M'}X_{M'}.$$ Then we can finish the proof by induction
as in the proof of Proposition \ref{4}.

When  $Q$ is of type $\widetilde{A}_{p,p}$, we can prove it by
Theorem \ref{difference}.
\end{proof}

\begin{Prop}\label{6}
Let $Q$ be a tame quiver with the \emph{alternating} orientation. If
$\mathrm{\underline{dim}}(T_1\oplus
R_1)=\mathrm{\underline{dim}}(T_2\oplus R_2)$ where $R_i$ are 0 or
any regular exceptional modules, $T_i$ are 0 or any indecomposable
regular modules with self-extension in non-homogeneous tubes and
there are no extension between $R_i$ and $T_i$, then
$$X_{T_1\oplus R_1}=X_{T_2\oplus R_2}+\sum_{\mathrm{\underline{dim}}R\prec
\mathrm{\underline{dim}}(T_2\oplus R_2)}a_{R}X_{R}$$ where
$a_{R}\in \mathbb{Q}$.
\end{Prop}
\begin{proof}
Suppose $\mathrm{\underline{dim}}(T_1\oplus
R_1)=(d_1,d_2,\cdots,d_n)$,  as in  the proof of Proposition
\ref{4}, we have
\begin{eqnarray*}
   && X^{d_1}_{S_1}X^{d_2}_{S_2}\cdots X^{d_n}_{S_n}\nonumber \\
   &=& (a_1X_{E^{(i)}_{1}}+\sum_{\mathrm{\underline{dim}}L'\prec
\mathrm{\underline{dim}}E^{(i)}_{1}}a_{L'}X_{L'})\cdots
(a_sX_{E^{(i)}_{s}}+\sum_{\mathrm{\underline{dim}}L''\prec
\mathrm{\underline{dim}}E^{(i)}_{s}}a_{L''}X_{L''})\nonumber \\
   && \times(a_{R_1}X_{R_1}+\sum_{\mathrm{\underline{dim}}L'''\prec
\mathrm{\underline{dim}}R_1}a_{L'''}X_{L'''})\nonumber \\
   &=& aX_{T_1\oplus R_1}+\sum_{\mathrm{\underline{dim}}L\prec
\mathrm{\underline{dim}}(T_1\oplus R_1)}a_{L}X_{L}.
\end{eqnarray*}
where $a\neq 0$ as the discussion in the proof of Proposition
\ref{4}. Similarly we have
\begin{eqnarray*}
   && X^{d_1}_{S_1}X^{d_2}_{S_2}\cdots X^{d_n}_{S_n} \nonumber \\
   &=& (b_1X_{E^{(j)}_{1}}+\sum_{\mathrm{\underline{dim}}M'\prec
\mathrm{\underline{dim}}E^{(j)}_{1}}b_{M'}X_{M'})\cdots
(b_tX_{E^{(j)}_{t}}+\sum_{\mathrm{\underline{dim}}M''\prec
\mathrm{\underline{dim}}E^{(j)}_{t}}b_{M''}X_{M''})\nonumber \\
   && \times(b_{R_2}X_{R_2}+\sum_{\mathrm{\underline{dim}}M'''\prec
\mathrm{\underline{dim}}R_2}b_{M'''}X_{M'''})\nonumber \\
   &=& bX_{T_2\oplus R_2}+\sum_{\mathrm{\underline{dim}}L\prec
\mathrm{\underline{dim}}(T_2\oplus R_2)}b_{M}X_{M}.
\end{eqnarray*}
where $b\neq 0.$ Thus
$$aX_{T_1\oplus R_1}+\sum_{\mathrm{\underline{dim}}L\prec
\mathrm{\underline{dim}}(T_1\oplus R_1)}a_{L}X_{L}=bX_{T_2\oplus
R_2}+\sum_{\mathrm{\underline{dim}}L\prec
\mathrm{\underline{dim}}(T_2\oplus R_2)}b_{M}X_{M}.$$
 Therefore by Proposition \ref{2} and Theorem \ref{d}, we have
$$X_{T_1\oplus R_1}=X_{T_2\oplus R_2}+\sum_{\mathrm{\underline{dim}}R\prec
\mathrm{\underline{dim}}(T_2\oplus R_2)}a_{R}X_{R}.$$
\end{proof}

\section{A $\mathbb{Z}$-basis for the cluster algebra of  affine type} In
this section, we will construct a $\bbz$-basis for the cluster
algebra of a tame quiver.
\begin{Definition}\cite{CK2005}
Let $Q$ be an acyclic quiver and $B=R-R^{tr}$. The quiver $Q$ is
called graded if there exists a linear form $\epsilon$ on
$\mathbb{Z}^{n}$ such that $\epsilon(B\alpha_i)<0$ for any $1\leq
i\leq n$ where $\alpha_i$ denotes the $i$-th vector of the
canonical basis of $\mathbb{Z}^{n}$.
\end{Definition}

\begin{Thm}\cite{CK2005}\label{40}
Let Q be a graded quiver and $\{M_1,\cdots,M_r\}$ a family objects
in $\mc(Q)$ such that $\mathrm{\underline{dim}}M_i\neq
\mathrm{\underline{dim}}M_j$ for $i\neq j$, then
$X_{M_1},\cdots,X_{M_r}$ are linearly independent over
$\mathbb{Q}$.
\end{Thm}
Now let $Q$ be a tame quiver with the \emph{alternating}
orientation. Note that the quiver $Q$ we consider is graded. Define
the set $\mathcal{S}(Q)$ to be
$$\{X_{L},X_{T\oplus R}|\mathrm{\underline{dim}}(T_1\oplus
R_1)\neq\mathrm{\underline{dim}}(T_2\oplus
R_2),\mathrm{Ext}^{1}_{\mc(Q)}(T,R)=0,\mathrm{Ext}^{1}_{\mc(Q)}(L,L)=0\}$$
where $L$ is any non-regular exceptional object, $R$ is $0$ or any
regular exceptional module and $T$ is $0$ or any indecomposable
regular module with self-extension.

To prove Theorem \ref{7}, Firstly we  need to prove the dimension
vectors of these objects associated to the corresponding elements
in $\mathcal{S}(Q)$ are different by Theorem \ref{40}.
\begin{Prop}\label{yy}
Let  $T\oplus R$ satisfy $X_{T\oplus R} \in \mathcal{S}(Q)$. If
$R\neq 0,$ then $\underline{\mathrm{dim}}M\neq m\delta$ for $m\in
\bbn$.
\end{Prop}
\begin{proof}
 Let $R_0$ be an indecomposable
regular exceptional module as a non-zero direct summand of $R$. We
set $M=T\oplus R=M'\oplus R_0$ and
$m\delta=\underline{\mathrm{dim}}M$. Since
$\mathrm{Ext}_{\mathcal{C}(Q)}^1(T,R)=0$ and
$\mathrm{Ext}_{\mathcal{C}(Q)}^1(R,R)=0$, we have
$$(\underline{\mathrm{dim}}M,\mathrm{\underline{dim}}R_0)
=(\underline{\mathrm{dim}}M',\mathrm{\underline{dim}}R_0)+(\underline{\mathrm{dim}}R_0,\mathrm{\underline{dim}}R_0)>0$$
where $(\alpha,\beta)=\left <\alpha,\beta\right
>+\left <\beta,\alpha\right >$ for $\alpha,\beta\in
\mathbb{Z}^{n}.$ It is a contradiction to
$(m\delta,\mathrm{\underline{dim}}R_0)=0$.
\end{proof}

\begin{Prop}\label{9}
Let $M$ be a  regular module associated to some element in
$\mathcal{S}(Q)$ and $L$ be a non-regular exceptional object in
$\mc(Q)$. Then $\underline{\mathrm{dim}}M\neq
\underline{\mathrm{dim}}L$.

\end{Prop}

\begin{proof}
If $L$ contains some $\tau P_i$ as its direct summand, we know
that
$$\underline{\mathrm{dim}}(\tau P_i)=(0,\cdots,0,-1,0,\cdots,0)$$
where the $i$-th component is $-1$. Suppose $L=\tau P_i\oplus\tau
P_{i_1}\cdots\tau P_{i_r}\oplus N$ where $N$ is an exceptional
module. Because  $L$ is an exceptional object, $ X_{\tau
P_i}X_{N}=X_{\tau P_i\oplus N}$ i.e.
$\mathrm{dim}_{\bbc}\mathrm{Hom}(P_i,N)=0$. Thus we have
$\mathrm{dim}_{\bbc} N(i)=0$ and $\mathrm{dim}_{\bbc}(\tau
P_i\oplus\tau P_{i_1}\cdots\tau P_{i_r}\oplus N)(i)\leq -1$.
However, $\underline{\mathrm{dim}}M\geq 0$. Therefore,
$\underline{\mathrm{dim}}M\neq \underline{\mathrm{dim}}L$.

If $L$ is a module. Suppose $\underline{\mathrm{dim}}M=
\underline{\mathrm{dim}}L$. Because $L$ is an exceptional module,
we know that $M$ belongs to the orbit of $L$ and then $M$ is a
degeneration of $L$. Hence, there exists some $\bbc Q$-module $U$
such that
$$ 0\longrightarrow
M\longrightarrow L\oplus U\longrightarrow U\longrightarrow 0$$ is
an exact sequence. Choose minimal $U$ so that we cannot separate
the following exact sequence
$$0\longrightarrow  0\longrightarrow U_1\longrightarrow U_1\longrightarrow
0$$ from the above short exact sequence. Thus  $M$ has a non-zero
map to every direct summand of $L$. Therefore $L$  has no
preprojective modules as direct summand because  $M$ is a regular
module.

Dually there exists a $\bbc Q$-module $V$ such that $$
0\longrightarrow V\longrightarrow V\oplus L\longrightarrow
M\longrightarrow 0$$ is an exact sequence. We can choose minimal
$V$ so that one cannot separate the following exact sequence
$$0\longrightarrow V_1\longrightarrow V_1\longrightarrow
0\longrightarrow 0$$ from the above short exact sequence. Thus
every direct summand of  $L$ has a non-zero map to $M$. Therefore
$L$  has no preinjective modules as direct summand because  M is a
regular module.

Therefore $L$ is a regular exceptional module, it is a
contradiction.
\end{proof}

Secondly, we need to prove that $\mathcal{S}(Q)$ is a
$\mathbb{Z}$-basis of $\mathcal{AH}(Q).$

\begin{Prop}\label{11}
$X_{M}X_{N}$ must be a $\mathbb{Z}$-combination of elements in the
set   $\mathcal{S}(Q)$ for any $M,N \in \mathcal{C}(Q).$
\end{Prop}

\begin{proof}
Let $M,N $ be in $\mathcal{C}(Q).$ By  Lemma \ref{induction}, we
know that $X_{M}X_{N}$ must be a $\mathbb{Q}$-linear combination
of elements in the set
$$\{X_{L},X_{T\oplus
R}|\mathrm{Ext}^{1}_{\mc(Q)}(T,R)=0,\mathrm{Ext}^{1}_{\mc(Q)}(L,L)=0\}$$
where $L$ is any non-regular exceptional object, $R$ is $0$ or any
regular exceptional module and $T$ is $0$ or any indecomposable
regular module with self-extension.

From those propositions in Section 4, we can easily find that
$X_{M}X_{N}$ is a $\mathbb{Q}$-linear combination of elements in the
set $\mathcal{S}(Q)$. Thus we have
$$X_{M}X_{N}=b_YX_{Y}+\sum_{\mathrm{\underline{dim}}Y'\prec
\mathrm{\underline{dim}}(M\oplus N)}b_{Y'}X_{Y'}$$ where
$\mathrm{\underline{dim}}Y=\mathrm{\underline{dim}}(M\oplus N)$,
$X_{Y},X_{Y'}\in \mathcal{S}(Q)$ and $b_Y,b_{Y'}\in \mathbb{Q}$.
Therefore by Proposition \ref{2} and Theorem \ref{d}, we have
$b_Y=1$. Note that there exists a partial order on these dimension
vectors by Definition \ref{p}. Thus in these remained $Y'$, we
choose these maximal elements denoted by $Y'_{1},\cdots,Y'_{s}$.
Then by $b_Y=1$ and the coefficients of Laurent expansions in
generalized cluster variables are integers, we obtain that
$b_{Y'_{1}},\cdots,b_{Y'_{s}}$ are integers. Using the same
method, we have $b_{Y'}\in \mathbb{Z}$.
\end{proof}

\begin{Thm}\label{7}
The set $\mathcal{S}(Q)$ is a  $\mathbb{Z}$-basis of
$\mathcal{AH}(Q).$
\end{Thm}
\begin{proof}
The proof follows from Proposition \ref{yy}, Proposition \ref{9}
and Proposition \ref{11}.
\end{proof}

\begin{Cor}\label{81}
$\mathcal{S}(Q)$ is a $\mathbb{Z}$-basis of the cluster algebra
$\mathcal{EH}(Q)$.
\end{Cor}
\begin{proof}
Let $E_1, \cdots, E_r$ be regular simple modules in a tube
$\mathcal{T}$ of rank $r>1.$ It is obvious that $X_{E_i}\in
\mathcal{EH}(Q)$. Then by   \cite[Proposition 6.2]{Dupont}, one can
show that $X_{E_i[n]}\in \mathcal{EH}(Q)$ for any $n\in \bbn$. Thus
by Proposition \ref{5}, we can prove that $X_{m\delta}\in
\mathcal{EH}(Q)$. By definition, $X_{L}\in \mathcal{EH}(Q)$ for $L$
satisfying $\mathrm{Ext}_{\mathcal{C}(Q)}^1(L,L)=0$. Thus
$\mathcal{S}(Q)\subset \mathcal{EH}(Q)$. It follows that
$\mathcal{S}(Q)$ is a $\mathbb{Z}$-basis of the cluster algebra
$\mathcal{EH}(Q)$ by Theorem \ref{7}.
\end{proof}
According to Theorem \ref{7} and Corollary \ref{81}, we have
\begin{Cor}\label{coincide}
Let $Q$ be an alternating tame quiver.
 Then $\mathcal{EH}(Q)=\mathcal{AH}(Q)$.
\end{Cor}

\begin{Prop}\label{12}
Let $Q$ be an alternating tame quiver with $Q_0=\{1,2,\cdots,
n\}.$ For any $\ud=(d_i)_{i\in Q_0}\in \bbz^n$, we have
$$\prod_{i=1}^{n}X^{d^+_i}_{S_i}X^{d^-_i}_{P_i[1]}=X_{M(d_1,\cdots,d_n)}+\sum_{\underline{dim}L\prec
(d_1,\cdots,d_n)}b_{L}X_{L}$$ where $X_{M(d_1,\cdots,d_n)}$ and
$X_{L}\in \mathcal{S}(Q)$,
$\mathrm{\underline{dim}}M=(d_1,\cdots,d_n)\in \mathbb{Z}^n$ and
$b_{L}\in \mathbb{Z}$.
\end{Prop}

\begin{proof}
It follows from Proposition \ref{2}, Theorem \ref{d} and Theorem
\ref{7}.
\end{proof}

Note that $\{X_{M(d_1,\cdots,d_n)}:(d_1,\cdots,d_n)\in
\mathbb{Z}^n\}$ is a $\mathbb{Z}$-basis of $\mathcal{AH}(Q)$, then
we have the following Corollary \ref{13} by Proposition \ref{12}.

\begin{Cor}\label{13}
The set $\{\prod_{i=1}^{n}X^{d^+_i}_{S_i}X^{d^-_i}_{P_i[1]}\mid
(d_1,\cdots,d_n)\in \mathbb{Z}^n\}$ is a $\mathbb{Z}$-basis of
$\mathcal{AH}(Q)$.
\end{Cor}

Suppose $Q$ is an acyclic quiver. Define the reflected quiver
$\sigma_{i}(Q)$ by reversing all the arrows ending at $i$. The
mutations can be viewed as generalizations of reflections i.e. if
$i$ is a sink or a source in $Q_0,$ then
$\mu_{i}(Q)=\sigma_{i}(Q)$ where $\mu_{i}$ denotes the mutation in
the direction $i$. Thus there is a natural isomorphism of cluster
algebras
$$\Phi: \mathcal{EH}(Q)\longrightarrow \mathcal{EH}(Q')$$
where $Q'$ is a quiver mutation equivalent to $Q$, and $\Phi$ is
called the canonical cluster algebras isomorphism.

Let $i$ be a sink in $Q_0$, $Q'=\sigma_{i}(Q)$ and
$R^{+}_{i}:\mathcal{C}({Q})\longrightarrow \mathcal{C}({Q'})$ be
the extended BGP-reflection functor defined in \cite{Zhu}. Denote
by $X^{Q}_{?}$ (resp. by $X^{\sigma_iQ}_{?}$) the Caldero-Chapton
map associated to $Q$ (resp. to $\sigma_iQ$).

Then the following hold.
\begin{Lemma}\cite{Zhu}\label{80}
Let $Q$ be an acyclic quiver and  $i$ be a sink in $Q$. Then
$R^{+}_{i}$ induces a triangle equivalence
 $$R^{+}_{i}:\mathcal{C}({Q})\longrightarrow
\mathcal{C}({\sigma_iQ})$$

\end{Lemma}

 We note that Lemma \ref{80} plays an essential importance to obtain the following lemmas.
\begin{Lemma}\cite{Dupont1}\label{50}
Let $Q$ be a tame quiver and  $i$ be a sink in $Q$. Denote by
$\Phi_i:\mathcal{A}(Q)\longrightarrow \mathcal{A}(\sigma_iQ)$ the
canonical cluster algebra isomorphism and by
$R^{+}_{i}:\mathcal{C}({Q})\longrightarrow \mathcal{C}({\sigma_iQ})$
the extended BGP-reflection functor. Then
$$\Phi_i(X^{Q}_{M})=X^{\sigma_iQ}_{R^{+}_{i}M}$$
where $M$ is any rigid object in $\mathcal{C}({Q})$ or any regular
module in non-homogeneous tubes.
\end{Lemma}

\begin{Lemma}\label{100}
Let $Q$ be a tame quiver and $E$ be  any regular simple module  of
dimension vector $\delta.$ Then we have
$\Phi_e(X^{\sigma_eQ}_{E})=X^{Q}_{R^+_e(E)}$ where $e$ is a sink or
source in $\sigma_eQ.$
\end{Lemma}
\begin{proof}
We only consider the case $e$ is a sink  in $\sigma_eQ.$ If $Q$ is
of type $\widetilde{A}_{p,q}$, then the lemma follows
\cite[Proposition 4.6]{Dupont1}. Now we assume that $Q$ is of type
$\widetilde{D}_n (n\geq 4)$ or $\widetilde{E}_m (m=6,7,8)$.
 Thus there are three non-homogeneous tubes
for $\sigma_eQ$ denoted by $\mathcal{T}_1, \mathcal{T}_2,
\mathcal{T}_3$. Let $E^{(i)}_1[n_i]$ be the unique indecomposable
regular module in $\mathcal{T}_i$ such that
$\mathrm{\underline{dim}}E^{(i)}_1[n_i]=\delta$ and
$\mathrm{reg}.\mathrm{top} (E^{(i)}_1[n_i])_{e}\neq 0$ for $i=1, 2,
3.$ Let $P_e$ and $I$ be $\sigma_eQ$-modules such that
$\mathrm{\underline{dim}}I=\delta-e$.  Using Theorem \ref{XX} (1),
we know that the product $2X_{P_e}^{\sigma_eQ}X^{\sigma_eQ}_{I}$ is
equal to
$$
-X^{\sigma_eQ}_{E}+\sum_{i=1}^3X^{\sigma_eQ}_{E^{(i)}_1[n_i]}
+\sum_{I', U\in S(\ud(I'))}\chi(\mathbb{P}\mathrm{Hom}(P_e, \tau
I)_{\str{U}\oplus I'[-1]})X^{\sigma_eQ}_{U\oplus I'[-1]}.
$$
Since $U\oplus I'[-1]$ is a preinjective module, applying the
isomorphism $\Phi_e$ to two side, by Lemma \ref{50}, we know that
the product $2X_{P_e[1]}^{Q}X^{Q}_{R_e^+(I)}$ is equal to
$$
-\Phi_e(X^{\sigma_eQ}_{E})+\sum_{i=1}^3X^{Q}_{R_e^+(E^{(i)}_1[n_i])}
+\sum_{I', U\in S(\ud(I'))}\chi(\mathbb{P}\mathrm{Hom}(P_e, \tau
I)_{\str{U}\oplus I'[-1]})X^{Q}_{R_e^+(U\oplus I'[-1])}.
$$
We note that $P^{Q}_e$ is the projective $\bbc Q$-module at $e$ and
$\mathrm{\underline{dim}}R^+_e(I)=\delta+e.$ Let $I^{Q}_e$ be the
simple injective module at $e$. Then any nonzero morphism from
$R^+_e(I)$ to $I^{Q}_e$ is epic and any nonzero morphism from
$P^{Q}_e$ to $R^+_eI$ is monic. Applying Theorem \ref{XX} (2) to the
product $X_{P^{Q}_e[1]}^{Q}X^{Q}_{R_e^+(I)},$ we obtain that the
product $2X_{P^{Q}_e[1]}^{Q}X^{Q}_{R_e^+(I)}$ is equal to
$$
-X^{Q}_{R^+_e(E)}+\sum_{i=1}^3X^{Q}_{R_e^+(E^{(i)}_1[n_i])}+\sum_{I',
U\in S(\ud(I'))}\chi(\mathbb{P}\mathrm{Hom}(P_e, \tau
I)_{\str{U}\oplus I'[-1]})X^{Q}_{R_e^+(U\oplus I'[-1])}.
$$
Hence, we have $\Phi_e(X^{\sigma_eQ}_{E})=X^{Q}_{R^+_e(E)}.$
\end{proof}

\begin{Thm}\label{60}
Let $Q$ be an alternating tame quiver  and $Q'=\sigma_{i_s}\cdots
\sigma_{i_1}(Q)$. Then a $\mathbb{Z}$-basis for the cluster
algebra of $Q'$ is the following set (denoted by
$\mathcal{S}(Q')$):
$$\{X_{L'},X_{T'\oplus R'}|\mathrm{\underline{dim}}(T'_1\oplus
R'_1)\neq\mathrm{\underline{dim}}(T'_2\oplus R'_2),
\mathrm{Ext}^{1}_{\mc(Q')}(T',R')=0,
\mathrm{Ext}^{1}_{\mc(Q')}(L',L')=0\}$$ where $L'$ is any
non-regular exceptional object, $R'$ is $0$ or any regular
exceptional module, $T'$ is $0$ or any indecomposable regular
module with self-extension.
\end{Thm}
\begin{proof}
If $Q'$ is a quiver of type $\widetilde{A}_{1,1}$, it is obvious
that
$$\{X_M,X_{n\delta}\mid M\in
\mathcal{C}(Q),\mathrm{Ext}^1(M,M)=0\}$$ is a $\mathbb{Z}$-basis
for cluster algebra of $\widetilde{A}_{1,1}$, which is called the
semicanonical basis in \cite{CZ}. If $Q'$ is not a quiver of type
$\widetilde{A}_{1,1}$ and $Q$ is an alternating tame quiver, then
in Theorem \ref{7}, we have already obtained a $\mathbb{Z}$-basis
for cluster algebra of $Q$,  denoted  by $\mathcal{S}(Q)$:
$$\{X_{L},X_{T\oplus R}|\mathrm{\underline{dim}}(T_1\oplus
R_1)\neq\mathrm{\underline{dim}}(T_2\oplus
R_2),\mathrm{Ext}^{1}_{\mc(Q)}(T,R)=0,\mathrm{Ext}^{1}_{\mc(Q)}(L,L)=0\}$$
where $L$ is any  non-regular exceptional object, $R$ is $0$ or any
regular exceptional module and $T$ is $0$ or any indecomposable
regular module with self-extension.  Thus $\Phi(\mathcal{S}(Q))$ is
a $\mathbb{Z}$-basis for the cluster algebra of $Q'$ because
$\Phi:\mathcal{EH}(Q)\longrightarrow \mathcal{EH}(Q')$ is the
canonical cluster algebras isomorphism. Then by Lemma \ref{50} and
Lemma \ref{100} we know that $\Phi(\mathcal{S}(Q))$ is exactly the
basis $\mathcal{S}(Q')$.
\end{proof}

\nd {\it Proof of Theorem 1.1.}  If $Q$ is reflection equivalent to
a quiver with alternating orientation, then the proof follows from
Theorem \ref{60}.  If any orientation of the underlying graph of $Q$
does not satisfy that any vertex is a sink or a source, then $Q$ is
a quiver of type $\widetilde{A}_{p,q}$ with $p\neq q$. The theorem
follows \cite[Theorem 4.8]{Dupont1}. \hspace{11cm}$\square$

For any dimension vector $\ud$, let $\ud=k\delta+\ud_0$ be the
canonical decomposition of $\ud$ (see \cite{Kac}) and $E$ be an
indecomposable regular simple module of dimension vector $\delta$.
There exists a regular rigid module $R$ of dimension vector $\ud_0$.
Then Theorem \ref{30} implies that the set
$$\mathcal{B}'(Q):=\{X_{L},X_{E[k]\oplus R}|L\in \mc(Q), \mathrm{Ext}^{1}_{\mc(Q)}(L,L)=0, \mathrm{Ext}^{1}_{\mc(Q)}(R,R)=0\}.$$
is a $\bbz$-basis for $\mathcal{EH}(Q).$

As in Corollary \ref{81}, let $E_1, \cdots, E_r$ be regular simple
modules in a tube $\mathcal{T}$ of rank $r>1.$ It is obvious that
$X_{E_i[n]}\in \mathcal{EH}(Q)$ for $1\leq i\leq r$ and  $n\in
\bbn$. Thus, we have
\begin{Cor}
Let $Q$ be a tame quiver. Then $\mathcal{EH}(Q)=\mathcal{AH}(Q)$.
\end{Cor}
Recall the set $\textbf{E}(Q)$  was defined in the introduction.
For any $\ud\in \bbz^{Q_0}$, we have
$$\prod_{i=1}^{n}X^{d^+_i}_{S_i}X^{d^-_i}_{P_i[1]}=b_MX_{M}+\sum_{\underline{dim}L\prec
(d_1,\cdots,d_n)}b_{L}X_{L}$$ where $X_{M}$ and $X_{L}\in
\mathcal{B}(Q)$, and $b_M, b_{L}\in \mathbb{Z}$. According to
Theorem \ref{XX}, we know that $b_M\neq 0$ and
$\mathrm{\mathrm{\underline{dim}}}M=(d_1,\cdots,d_n)\in
\mathbb{Z}^n$.  Thus we have the following corollary.
\begin{Cor}\label{dimension}
Let $Q$ be a tame quiver. Then $\bbz^{Q_0}=\textbf{D}(Q)\cup
\textbf{E}(Q).$
\end{Cor}

\section{Examples}
In this section, we will give some examples for special cases to
explain the $\mathbb{Z}$-bases explicitly.

\begin{subsection}{$\bbz$-basis for finite type}
For finite type, we know that there are no regular modules in
mod-$\bbc Q$. Thus the $\mathbb{Z}$-bases are exactly
$$\{X_M\mid M\in \mathcal{C}(Q),\mathrm{Ext}^1(M,M)=0\}$$ in
\cite{CK2005}.
\end{subsection}

\begin{subsection}{$\bbz$-basis for the Kronecker quiver}
Consider the Kronecker quiver $K$. Let $M$ be a regular simple $\bbc
K$-module and $M[n]$ be the regular module with regular socle $M$
and regular length $n\in \bbn$. Since $X_{M[n]}=X_{M'[n]}$ for any
regular simple $M'$. We set $X_{n\delta}:=X_{M[n]}.$ In this case,
there is a $\mathbb{Z}$-basis of $\mc(K)$ is
$$\mathcal{B}'(K)=\{X_M,X_{n\delta}\mid M\in \mathcal{C}(Q),\mathrm{Ext}^1(M,M)=0\}$$ which
is called the semicanonical basis in \cite{CZ}. If we modify:
$$z_1:=X_{\delta},\ z_n:=X_{n\delta}-X_{(n-2)\delta}.$$
Then $\{X_M,z_n\mid M\in \mathcal{C}(Q),\mathrm{Ext}^1(M,M)=0\}$
is the canonical basis for cluster algebra of Kronecker quiver in
\cite{SZ}. \end{subsection}

\begin{subsection}{$\bbz$-basis for $\widetilde{D}_4$}\label{d4}
Let $Q$ be the tame quiver of type $\widetilde{D}_4$ as follows
$$
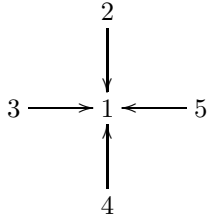
\xymatrix{& 2 \ar[d] &\\
3 \ar[r] & 1  & 5 \ar[l]\\
& 4 \ar[u] &}
$$
We denote the minimal imaginary root by $\delta=(2,1,1,1,1)$. The
regular simple modules of dimension vector $\delta$ are
$$
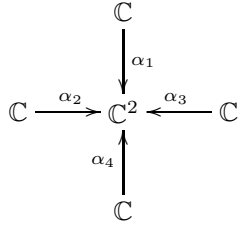
\xymatrix{& \bbc \ar[d]^{\alpha_1} &\\
\bbc \ar[r]^{\alpha_2} & \bbc^2  & \bbc \ar[l]_{\alpha_3}\\
& \bbc \ar[u]^{\alpha_4} &}
$$
with linear maps
$$
\alpha_1=\left(%
\begin{array}{c}
  1 \\
  0 \\
\end{array}%
\right), \alpha_2=\left(%
\begin{array}{c}
  0 \\
  1 \\
\end{array}%
\right),\alpha_3=\left(%
\begin{array}{c}
  1 \\
  1 \\
\end{array}%
\right),\alpha_4=\left(%
\begin{array}{c}
  \lambda\\
  \mu \\
\end{array}%
\right)
$$
where $\lambda/\mu\in \mathbb{P}^1, \lambda/\mu\neq 0,1,\infty.$
Let $M$ be any regular simple $\bbc Q$-module of dimension vector
$\delta$ and  $X_M$ be the generalized cluster variable associated
to $M$ by the reformulation of the Caldero-Chapton map. Then we
have
\begin{Prop}\cite{DX1}
$X_{M}=\frac{1}{x^2_1x_2x_3x_4x_5}+\frac{4}{x_1x_2x_3x_4x_5}+
\frac{x^2_1+4x_1+6}{x_2x_3x_4x_5}+\frac{x_2x_3x_4x_5+2}{x^2_1}+\frac{4}{x_1}.$
\end{Prop}
We define $X_{n\delta}:=X_{M[n]}$ for $n\in \bbn.$ Now, we
consider three non-homogeneous tubes labelled by the subset
$\{0,1,\infty\}$ of $\mathbb{P}^1$. The regular simple modules in
non-homogeneous tubes are denoted by $E_1,E_2,E_3,E_4,E_5,E_6$,
where
$$\underline{\mathrm{dim}}E_1=(1,1,1,0,0),\,\,\underline{\mathrm{dim}}E_2=(1,0,0,1,1),\,\,
\underline{\mathrm{dim}}E_3=(1,1,0,1,0),$$$$
\underline{\mathrm{dim}}E_4=(1,0,1,0,1),\,\,
\underline{\mathrm{dim}}E_5=(1,0,1,1,0),\,\,\underline{\mathrm{dim}}E_6=(1,1,0,0,1).$$
We note that $\{E_1, E_2\}$, $\{E_3, E_4\}$ and $\{E_5, E_6\}$ are
pairs of the regular simple modules at the mouth of non-homogeneous
tubes labelled by $1,\infty$ and $0$, respectively. We set
$X_{n\delta_i}:= X_{E_i[2n]}$ for  $1\leq i\leq 6$.
\begin{Prop}\cite{DX1}\label{d4-delta}
For any $n\in \bbn$, we have
$X_{n\delta_1}=X_{n\delta_2}=X_{n\delta_3}=X_{n\delta_4}=X_{n\delta_5}=X_{n\delta_6}$
and $X_{n\delta_1}=X_{n\delta}+X_{(n-1)\delta}$ where $X_0=1$.
\end{Prop}

\begin{Thm}\cite{DX1}
A $\mathbb{Z}$-basis for cluster algebra of $\widetilde{D}_4$ is
the following set denoted by $\mathcal{B}(Q)$:
$$\{
X_{L},X_{m\delta},X_{E_i[2k+1]\oplus
R}|\mathrm{\underline{dim}}(E_i[2k+1]\oplus
R_1)\neq\mathrm{\underline{dim}}(E_j[2l+1]\oplus R_2),$$
$$\mathrm{Ext}_{\mathcal{C}(Q)}^1(L,L)=0,\mathrm{Ext}_{\mathcal{C}(Q)}^1(E_i[2k+1],R)=0,m,k,l\geq
0,1\leq i,j\leq 6\}$$ where $L$ is any non-regular exceptional
object, $R$ is $0$ or any regular exceptional module.
\end{Thm}
\end{subsection}
\section{Inductive multiplication formulas for a tube}
In this section, let $Q$ be any tame quiver. Now we fix a tube with
rank $r$ and these regular simple modules are $E_1,\cdots,E_r$ with
$\tau E_{2}=E_{1}, \cdots, \tau E_1=E_r.$ Let $X_{E_i}$ be the
corresponding generalized cluster variable for $1 \leq i \leq r.$
With these notations, we have the following inductive cluster
multiplication formulas.
\begin{Thm}\label{16}
Let $i,j, k,l,m$ and $r$ be in $\bbz$ such that $1\leq k\leq
mr+l,0\leq l\leq r-1,1\leq i,j\leq r,m\geq 0$.

\nd (1)When $j\leq i$, then

1)for $k+i\geq r+j$, we have
$X_{E_i[k]}X_{E_j[mr+l]}=X_{E_i[(m+1)r+l+j-i]}X_{E_j[k+i-r-j]}+X_{E_i[r+j-i-1]}X_{E_{k+i+1}[(m+1)r+l+j-k-i-1]},$

2)for $k+i< r+j$ and $i\leq l+j\leq k+i-1$, we have
$X_{E_i[k]}X_{E_j[mr+l]}=X_{E_j[mr+k+i-j]}X_{E_i[l+j-i]}+X_{E_j[mr+i-j-1]}X_{E_{l+j+1}[k+i-l-j-1]},$

3)for other conditions, we have
$X_{E_i[k]}X_{E_j[mr+l]}=X_{E_i[k]\oplus E_j[mr+l]}$.

\nd (2)When $j> i$, then

1)for $k\geq j-i,$ we have
$X_{E_i[k]}X_{E_j[mr+l]}=X_{E_i[j-i-1]}X_{E_{k+i+1}[mr+l+j-k-i-1]}+X_{E_i[mr+l+j-i]}X_{E_j[k+i-j]},$

2)for $k< j-i$ and $i\leq l+j-r\leq k+i-1$, we have
$X_{E_i[k]}X_{E_j[mr+l]}=X_{E_j[(m+1)r+k+i-j]}X_{E_i[l+j-r-i]}+X_{E_j[(m+1)r+i-j-1]}X_{E_{l+j+1}[k+r+i-l-j-1]},$

3)for other conditions, we have
$X_{E_i[k]}X_{E_j[mr+l]}=X_{E_i[k]\oplus E_j[mr+l]}.$
\end{Thm}

\begin{proof}
We only prove (1) and (2) is totally similar to (1).

1) When $k=1,$ by $k+i\geq r+j$ and $1\leq j\leq i\leq
r\Longrightarrow i=r$ and $j=1.$\\
Then by the cluster multiplication theorem in Theorem \ref{XX} or
Theorem \ref{CK}, we have
$$X_{E_r}X_{E_1[mr+l]}=X_{E_r[mr+l+1]}+X_{E_2[mr+l-1]}.$$
 When $k=2,$ by $k+i\geq r+j$ and $1\leq j\leq i\leq
r\Longrightarrow
i=r\ or\ i=r-1.$\\
For $i=r\Longrightarrow j=1\ or\ j=2$:\\
The case for $i=r$ and $j=1$, we have
\begin{eqnarray}
    && X_{E_r[2]}X_{E_1[mr+l]}  \nonumber\\
   &=& (X_{E_r}X_{E_1}-1)X_{E_1[mr+l]} \nonumber\\
  &=& X_{E_1}(X_{E_r[mr+l+1]}+X_{E_2[mr+l-1]})-X_{E_1[mr+l]} \nonumber\\
  &=& X_{E_1}X_{E_r[mr+l+1]}+(X_{E_1[mr+l]}+X_{E_3[mr+l-2]})-X_{E_1[mr+l]} \nonumber\\
  &=& X_{E_1}X_{E_r[mr+l+1]}+X_{E_3[mr+l-2]}.\nonumber
\end{eqnarray}
The case for $i=r$ and $j=2$, we have
\begin{eqnarray}
&& X_{E_r[2]}X_{E_2[mr+l]} \nonumber\\
  &=& (X_{E_r}X_{E_1}-1)X_{E_2[mr+l]} \nonumber\\
  &=& X_{E_r}(X_{E_1[mr+l+1]}+X_{E_3[mr+l-1]})-X_{E_2[mr+l]} \nonumber\\
  &=& X_{E_r[mr+l+2]}+(X_{E_2[mr+l]}+X_{E_r}X_{E_3[mr+l-1]})-X_{E_2[mr+l]} \nonumber\\
  &=& X_{E_r[mr+l+2]}+X_{E_r}X_{E_3[mr+l-1]}.\nonumber
\end{eqnarray}

For $i=r-1\Longrightarrow j=1$:
\begin{eqnarray}
 && X_{E_{r-1}[2]}X_{E_1[mr+l]} \nonumber\\
  &=& (X_{E_{r-1}}X_{E_r}-1)X_{E_1[mr+l]} \nonumber\\
  &=& X_{E_{r-1}}(X_{E_r[mr+l+1]}+X_{E_2[mr+l-1]})-X_{E_1[mr+l]} \nonumber\\
  &=& (X_{E_{r-1}[mr+l+2]}+X_{E_1[mr+l]})+X_{E_{r-1}}X_{E_2[mr+l-1]}-X_{E_1[mr+l]} \nonumber\\
  &=& X_{E_{r-1}[mr+l+2]}+X_{E_{r-1}}X_{E_2[mr+l-1]}.\nonumber
\end{eqnarray}\\
Now, suppose it holds for $k\leq n,$ then by induction we have
\begin{eqnarray*}
   && X_{E_i[n+1]}X_{E_j[mr+l]}\nonumber \\
   &=& (X_{E_i[n]}X_{E_{i+n}}-X_{E_i[n-1]})X_{E_j[mr+l]}\nonumber \\
   &=& X_{E_{i+n}}(X_{E_i[n]}X_{E_j[mr+l]})-X_{E_i[n-1]}X_{E_j[mr+l]} \nonumber\\
   &=& X_{E_{i+n}}(X_{E_i[(m+1)r+l+j-i]}X_{E_j[n+i-r-j]}+X_{E_i[r+j-i-1]}X_{E_{n+i+1}[(m+1)r+l+j-n-i-1]})\nonumber \\
   && -(X_{E_i[(m+1)r+l+j-i]}X_{E_j[n+i-r-j-1]}+X_{E_i[r+j-i-1]}X_{E_{n+i}[(m+1)r+l+j-n-i]}) \nonumber\\
   &=& X_{E_i[(m+1)r+l+j-i]}(X_{E_j[n+i+1-r-j]}+X_{E_j[n+i-r-j-1]}) \nonumber\\
   && +X_{E_i[r+j-i-1]}(X_{E_{n+i}[(m+1)r+l+j-n-i]}
+X_{E_{n+i+2}[(m+1)r+l+j-n-i-2]}) \nonumber\\
   && -(X_{E_i[(m+1)r+l+j-i]}X_{E_j[n+i-r-j-1]}+X_{E_i[r+j-i-1]}X_{E_{n+i}[(m+1)r+l+j-n-i]})\nonumber \\
   &=& X_{E_i[(m+1)r+l+j-i]}X_{E_j[n+i+1-r-j]}+X_{E_i[r+j-i-1]}X_{E_{n+i+2}[(m+1)r+l+j-n-i-2]}.
\end{eqnarray*}

2)  When $k=1,$ by $i\leq l+j\leq k+i-1\Longrightarrow i\leq
l+j\leq i\Longrightarrow i=l+j.$\\
Then by  Theorem \ref{XX} or Theorem \ref{CK}, we have
$$X_{E_{i}}X_{E_j[mr+l]}=X_{E_{l+j}}X_{E_j[mr+l]}=X_{E_j[mr+l+1]}+X_{E_j[mr+l-1]}$$
  When $k=2,$ by $i\leq l+j\leq k+i-1\Longrightarrow i\leq
l+j\leq i+1\Longrightarrow i=l+j\ or\ i+1=l+j$:\\
For $i=l+j$, we have
\begin{eqnarray*}
   && X_{E_{i}[2]}X_{E_j[mr+l]} \nonumber\\
   &=& X_{E_{l+j}[2]}X_{E_j[mr+l]}\nonumber \\
   &=& (X_{E_{l+j}}X_{E_{l+j+1}}-1)X_{E_j[mr+l]}\nonumber \\
   &=& (X_{E_j[mr+l+1]}+X_{E_j[mr+l-1]})X_{E_{l+j+1}}-X_{E_j[mr+l]}\nonumber \\
   &=& X_{E_j[mr+l+2]}+X_{E_j[mr+l]}+X_{E_{l+j+1}}X_{E_j[mr+l-1]}-X_{E_j[mr+l]}\nonumber \\
   &=& X_{E_j[mr+j+2]}+X_{E_{l+j+1}}X_{E_j[mr+l-1]}.
\end{eqnarray*}

For $i+1=l+j$, we have
\begin{eqnarray*}
   && X_{E_{i}[2]}X_{E_j[mr+l]}\nonumber \\
   &=& X_{E_{l+j-1}[2]}X_{E_j[mr+l]}\nonumber \\
   &=& (X_{E_{l+j-1}}X_{E_{l+j}}-1)X_{E_j[mr+l]}\nonumber \\
   &=& (X_{E_j[mr+l+1]}+X_{E_j[mr+l-1]})X_{E_{l+j-1}}-X_{E_j[mr+l]} \nonumber \\
   &=& X_{E_j[mr+l+1]}X_{E_{l+j-1}}+(X_{E_j[mr+l]}+X_{E_j[mr+l-2]})-X_{E_j[mr+l]}\nonumber \\
   &=& X_{E_j[mr+l+1]}X_{E_{l+j-1}}+X_{E_j[mr+l-2]}.
\end{eqnarray*}

Suppose it holds for $k\leq n,$ then by induction we have
\begin{eqnarray*}
   && X_{E_i[n+1]}X_{E_j[mr+l]}\nonumber \\
   &=& (X_{E_i[n]}X_{E_{i+n}}-X_{E_i[n-1]})X_{E_j[mr+l]}\nonumber \\
   &=& (X_{E_i[n]}X_{E_j[mr+l]})X_{E_{i+n}}-X_{E_i[n-1]}X_{E_j[mr+l]}\nonumber\\
   &=& (X_{E_j[mr+n+i-j]}X_{E_i[l+j-i]}+X_{E_j[mr+i-j-1]}X_{E_{l+j+1}[n+i-l-j-1]})X_{E_{i+n}} \nonumber\\
   && -(X_{E_j[mr+n+i-j-1]}X_{E_i[l+j-i]}+X_{E_j[mr+i-j-1]}X_{E_{l+j+1}[n+i-l-j-2]})\nonumber \\
   &=& (X_{E_j[mr+n+i+1-j]}+X_{E_j[mr+n+i-j-1]})X_{E_i[l+j-i]} \nonumber\\
   && +(X_{E_{l+j+1}[n+i-l-j]}+X_{E_{l+j+1}[n+i-l-j-2]})X_{E_j[mr+i-j-1]} \nonumber\\
   && -(X_{E_j[mr+n+i-j-1]}X_{E_i[l+j-i]}+X_{E_j[mr+i-j-1]}X_{E_{l+j+1}[n+i-l-j-2]})\nonumber \\
   &=& X_{E_j[mr+n+i+1-j]}X_{E_i[l+j-i]}+X_{E_j[mr+i-j-1]}X_{E_{l+j+1}[n+i-l-j]}.
\end{eqnarray*}

3)  It is trivial by the definition of the  Caldero-Chapton map.
\end{proof}

We explain Theorem \ref{16} by the following proposition involving
a tube of rank $3$.

\begin{Prop}\label{18}
(1) For $n\geq 3m+1$, then
$$X_{E_2[3m+1]}X_{E_1[n]}=X_{E_2}X_{E_1[n+3m]}+X_{E_1[n+3m-3]}+X_{E_2}X_{E_1[n+3m-6]}+X_{E_1[n+3m-9]}$$
$$\hspace{2.3cm}+\cdots+X_{E_2}X_{E_1[n-3m+6]}+X_{E_1[n-3m+3]}
+X_{E_2}X_{E_1[n-3m]},$$

(2) For $n\geq 3m+2$, then
$$X_{E_2[3m+2]}X_{E_1[n]}=X_{E_2[n+3m+2]}+X_{E_2}X_{E_2[n+3m-1]}+X_{E_2[n+3m-4]}+X_{E_2}X_{E_2[n+3m-7]}$$
$$\hspace{-0.3cm}+\cdots+X_{E_2[n-3m+2]}+X_{E_2}X_{E_2[n-3m-1]},$$

(3) For $n\geq 3m+3$, then
$$\hspace{-0.56cm}X_{E_2[3m+3]}X_{E_1[n]}=X_{E_1[n+3m+3]}+X_{E_3[n+3m+1]}+X_{E_2[n+3m-1]}+X_{E_1[n+3m-3]}$$
$$\hspace{1.1cm}+X_{E_3[n+3m-5]}+X_{E_2[n+3m-7]}+\cdots+X_{E_1[n-3m+3]}$$
$$\hspace{0.35cm}+X_{E_3[n-3m+1]}+X_{E_2[n-3m-1]}+X_{E_1[n-3m-3]}.$$
\end{Prop}
\begin{proof}
(1)  By Theorem \ref{16}, then we have
\begin{eqnarray*}
   && X_{E_2[3m+1]}X_{E_1[n]}\nonumber \\
   &=& X_{E_2[n+2]}X_{E_1[3m-1]}+X_{E_2}X_{E_1[n-3m]}\nonumber \\
   &=& X_{E_2[3m-2]}X_{E_1[n+3]}+X_{E_1[n-3m+3]}+X_{E_2}X_{E_1[n-3m]} \nonumber\\
   &=& X_{E_2[n+5]}X_{E_1[3m-4]}+X_{E_2}X_{E_1[n-3m+6]}+X_{E_1[n-3m+3]}
+X_{E_2}X_{E_1[n-3m]}\nonumber \\
   &=& X_{E_2[3m-5]}X_{E_1[n+6]}+X_{E_1[n-3m+9]}+X_{E_2}X_{E_1[n-3m+6]}+X_{E_1[n-3m+3]}
+X_{E_2}X_{E_1[n-3m]}\nonumber \\
   && \cdots\nonumber \\
   &=& X_{E_2}X_{E_1[n+3m]}+X_{E_1[n+3m-3]}+X_{E_2}X_{E_1[n+3m-6]}+X_{E_1[n+3m-9]}\nonumber \\
   && +\cdots+X_{E_2}X_{E_1[n-3m+6]}+X_{E_1[n-3m+3]}
+X_{E_2}X_{E_1[n-3m]}.
\end{eqnarray*}

(2) By  Theorem \ref{16}, then we have
\begin{eqnarray*}
   && X_{E_2[3m+2]}X_{E_1[n]}\nonumber \\
   &=& X_{E_2[n+2]}X_{E_1[3m]}+X_{E_2}X_{E_2[n-3m-1]}\nonumber \\
   &=& X_{E_2[3m-1]}X_{E_1[n+3]}+X_{E_2[n-3m+2]}+X_{E_2}X_{E_2[n-3m-1]}\nonumber \\
   &=& X_{E_2[n+5]}X_{E_1[3m-3]}+X_{E_2}X_{E_2[n-3m+5]}+X_{E_2[n-3m+2]}+X_{E_2}X_{E_2[n-3m-1]}\nonumber \\
   &=& X_{E_2[3m-4]}X_{E_1[n+6]}+X_{E_2[n-3m+8]}+X_{E_2}X_{E_2[n-3m+5]}+X_{E_2[n-3m+2]}+X_{E_2}X_{E_2[n-3m-1]}\nonumber \\
   && \cdots \nonumber\\
   &=& X_{E_2[n+3m+2]}+X_{E_2}X_{E_2[n+3m-1]}+X_{E_2[n+3m-4]}+X_{E_2}X_{E_2[n+3m-7]}\nonumber \\
   && +\cdots+X_{E_2[n-3m+2]}+X_{E_2}X_{E_2[n-3m-1]}.
\end{eqnarray*}

(3)  By  Theorem \ref{16}, then we have
\begin{eqnarray*}
   && X_{E_2[3m+3]}X_{E_1[n]}\nonumber \\
   &=& X_{E_2[n+2]}X_{E_1[3m+1]}+X_{E_2}X_{E_3[n-3m-2]}\nonumber \\
   &=& X_{E_2[n+2]}X_{E_1[3m+1]}+X_{E_2[n-3m-1]}+X_{E_1[n-3m-3]}\nonumber \\
   &=& X_{E_2[3m]}X_{E_1[n+3]}+X_{E_3[n-3m+1]}+X_{E_2[n-3m-1]}+X_{E_1[n-3m-3]}\nonumber \\
   && \cdots \nonumber\\
   &=& X_{E_1[n+3m+3]}+X_{E_3[n+3m+1]}+X_{E_2[n+3m-1]}\nonumber \\
   && +X_{E_3[n+3m-5]}+X_{E_2[n+3m-7]}+\cdots+X_{E_1[n-3m+3]} \nonumber\\
   && +X_{E_3[n-3m+1]}+X_{E_2[n-3m-1]}+X_{E_1[n-3m-3]}.
\end{eqnarray*}
\end{proof}
\begin{Cor}\label{20}
When $n=3k+1$,we can rewrite Proposition \ref{18} as following:

(1) For $n\geq 3m+1$, then
$$\hspace{-1.8cm}X_{E_2[3m+1]}X_{E_1[n]}=X_{E_1[n+3m+1]}+X_{E_1[n+3m-1]}+X_{E_1[n+3m-3]}+\cdots$$
$$\hspace{2.4cm}+X_{E_1[n-3m+5]}+X_{E_1[n-3m+3]}+X_{E_1[n-3m+1]}
+X_{E_1[n-3m-1]},$$

(2) For $n\geq 3m+2$, then
$$\hspace{-0.6cm}X_{E_2[3m+2]}X_{E_1[n]}=X_{E_2[n+3m+2]}+X_{E_2[n+3m]}+X_{E_2[n+3m-2]}+X_{E_2[n+3m-4]}$$
$$\hspace{0.6cm}+\cdots+X_{E_2[n-3m+2]}+X_{E_2[n-3m]}+X_{E_2[n-3m-2]},$$

(3) For $n\geq 3m+3$, then
$$\hspace{-2.7cm}X_{E_2[3m+3]}X_{E_1[n]}=X_{E_1[n+3m+3]}+X_{E_3[n+3m+1]}+X_{E_2[n+3m-1]}$$
$$\hspace{0.8cm}+X_{E_1[n+3m-3]}
+X_{E_3[n+3m-5]}+X_{E_2[n+3m-7]}+\cdots$$
$$\hspace{2.4cm}+X_{E_1[n-3m+3]}+X_{E_3[n-3m+1]}+X_{E_2[n-3m-1]}+X_{E_1[n-3m-3]}.$$
\end{Cor}
\begin{proof}
When $n=3k+1$,we have the following equations:
$$X_{E_2}X_{E_1[n]}=X_{E_1[n+1]}+X_{E_1[n-1]}$$
$$X_{E_2}X_{E_2[n-1]}=X_{E_2[n]}+X_{E_2[n-2]}$$
Then the proof  immediately follows from Proposition \ref{18}.
\end{proof}

In the same way by using Theorem \ref{16}, we have the following
proposition.
\begin{Prop}\label{19}
(1) For $n\geq 3m+1$, then
$$X_{E_1[3m+1]}X_{E_1[n]}=X_{E_1}X_{E_1[n+3m]}+X_{E_1[n+3m-3]}+X_{E_1}X_{E_1[n+3m-6]}+X_{E_1[n+3m-9]}$$
$$\hspace{1.7cm}+\cdots+X_{E_1}X_{E_1[n-3m+6]}+X_{E_1[n-3m+3]}
+X_{E_1}X_{E_1[n-3m]},$$

(2) For $n\geq 3m+2$, then
$$\hspace{-1.5cm}X_{E_1[3m+2]}X_{E_1[n]}=X_{E_1[2]}(X_{E_1[n+3m]}+X_{E_1[n+3m-6]}+\cdots+X_{E_1[n-3m]}),$$

(3) For $n\geq 3m+3$, then
$$X_{E_1[3m+3]}X_{E_1[n]}=X_{E_1[n+3m+3]}+X_{E_2}X_{E_1[n+3m]}+X_{E_1[n+3m-3]}+X_{E_2}X_{E_1[n+3m-6]}$$
$$\hspace{1.0cm}+\cdots+X_{E_1[n-3m+3]}+X_{E_2}X_{E_1[n-3m]}+X_{E_1[n-3m-3]}.$$
\end{Prop}
\begin{Remark}
(1) By the same method in Corollary \ref{20}, we can also rewrite
Proposition \ref{19} if we consider these different n. Here we
omit it.

(2) In fact, we can also check Proposition \ref{18}, Corollary
\ref{20} and Proposition \ref{19} by induction.
\end{Remark}

\section{Appendix on the difference property}

In this appendix, we show that the difference property holds for any
tame quiver without oriented cycles.

\begin{subsection}{}Let $Q$ be a tame quiver without oriented cycles and $\delta$
be its minimal imaginary root. For any $\bbc Q$-module $M$, define
$$\partial(M):=\langle \delta, \mathrm{\underline{dim}} M\rangle.$$ Let $(\mathrm{\underline{dim}} M)_i$ be the $i$-th component of $\mathrm{\underline{dim}} M$ for $i\in Q_0.$
It is called the defect of $M$. A pair $(X, Y)$ of $\bbc Q$-modules is called an
orthogonal exceptional pair if $X$ and $Y$ are exceptional modules
such that
$$
\mathrm{Hom}_{\bbc Q}(X, Y)=\mathrm{Hom}_{\bbc Q}(Y,
X)=\mathrm{Ext}^1_{\bbc Q}(Y, X)=0.
$$
We denote by $\mathcal{S}(A, B)$ the full subcategory of
$\mathrm{mod} \bbc Q$ generated by $X$ and $Y$.
  The following result is well-known.
\begin{Lemma}
Let $Q$ be a tame quiver. We fix a preprojective module $P$ and a
preinjective module $I$ such that $\partial(P)=-1$ and
$\mathrm{\underline{dim}}P+\mathrm{\underline{dim}}I=\delta$. Then
the subcategory $\mathcal{S}(P, I)$ is a hereditary abelian
subcategory and there is an equivalence
$$
\mathcal{F}:\mathrm{ mod} K\rightarrow \mathcal{S}(P, I)
$$
where $K$ is the Kronecker algebra $\bbc (\xymatrix{1\ar @<2pt>[r]
\ar@<-2pt>[r]& 2})$. Every tube of $\mathrm{mod} \bbc Q$ contains
a unique module $M$ as the middle term of the extension of $P$ by
$I.$
\end{Lemma}
Let $e\in Q_0$ be an extending vertex, i.e., $\delta_e=1$. Let $P_e$
(resp. $I_e$) be the projective module (resp. injective module)
corresponding to  $e$. Let $Q'$ be the subquiver of $Q$ deleting the
vertex $e$ and involving edges. Then $Q'$ is a Dynkin quiver. Let
$I$ (resp. $P$) be the unique indecomposable preinjective (resp.
preprojective) $\bbc Q$-module of dimension vector
$\delta-\mathrm{\underline{dim}} P_e $ (resp.
$\delta-\mathrm{\underline{dim}} I_e$). We obtain a subcategory
$\mathcal{S}(P_e, I)$ (resp. $\mathcal{S}(P, I_e)$). Consider the
universal extensions
$$
0\rightarrow P^2_e\rightarrow L\rightarrow I\rightarrow 0 \, \,
\mbox{ and }\, \, 0\rightarrow P\rightarrow L'\rightarrow
I^2_e\rightarrow 0.
$$
Here, $L$ (resp. $L'$) is unique preprojective (resp. preinjective)
$\bbc Q$-module with dimension vector $\delta
+\mathrm{\underline{dim}} P_e$ (resp.
$\delta+\mathrm{\underline{dim}} I_e$). Then we obtain a pair $(I,
T^{-1}L)$ as the mutation of the pair $(P_e, I)$ where $T$ is the
shift functor in $\md^b(\bbc Q).$ Indecomposable regular $\bbc
Q$-modules consist of tubes indexed by the projective line
$\mathbb{P}^1.$ As in \cite[Section 9]{CB}, each tube contains a
unique module in the set
$$\Omega=\{\mbox{iso.classes of indecomposable } X \mid
\mathrm{\underline{dim}} X=\delta, (\mathrm{\underline{dim}}
\mathrm{reg.top}(X))_e\neq 0\}.$$ There are the following
bijections:
$$
\xymatrix{\mathbb{P}\mathrm{Hom} (P_e, L)\ar[r] &\Omega &
\mathbb{P}\mathrm{Ext}^1(I, P_e)\ar[l]}
$$
induced by the map sending $0\neq \theta\in \mathrm{Hom} (P_e, L)$
to $\mathrm{Coker} \theta$ and the map sending $\varepsilon\in
\mathrm{Ext}^1(I, P_e)$ to the isomorphism class of the middle term.
We note that the bijection between $\mathbb{P}\mathrm{Hom} (P_e, L)$
and $\mathbb{P}\mathrm{Ext}^1(I, P_e)$ is induced by the universal
extension. Note that $0\neq \theta\in \mathrm{Hom} (P_e, L)$ is
mono. Dually, we also have an orthogonal exceptional pair $(P,
I_e)$. Each tube contains a unique module in the set
$$\Omega'=\{\mbox{isoclasses of indecomposable } X \mid
\mathrm{\underline{dim}} X=\delta, (\mathrm{\underline{dim}}
\mathrm{reg.socle}(X))_e\neq 0\}.$$ There are the following
bijections:
$$
\xymatrix{\mathbb{P}\mathrm{Hom} (L', I_e)\ar[r] &\Omega' &
\mathbb{P}\mathrm{Ext}^1(I_e, P)\ar[l]}
$$
induced by the map sending $0\neq \theta'\in \mathrm{Hom} (L', I_e)$
to $\mathrm{ker} \theta'$ and the map sending $\varepsilon'\in
\mathrm{Ext}^1(I_e, P)$ to the isomorphism class of the middle term.
Note that $0\neq \theta'\in \mathrm{Hom} (L', I_e)$ is epic.

Now we assume $e$ is a sink in the quiver $Q$. Let $\mathcal{T}$ be
any tube of rank $n>1$ and $E_1, \cdots, E_n$ be the regular simple
modules in $\mathcal{T}$. Without loss of generality, we assume that
$\tau E_i=E_{i-1}$ for $i=2, \cdots n$, $\tau E_1=E_n$ and
$(E_n)_{e}\neq 0.$ Let $E_i[r]$ be the indecomposable  regular
module in $\mathcal{T}$ with quasi-socle $E_i$ and quasi-length $r$
for $i=1, \cdots, n.$ Then $E_1[n]\in \Omega.$ It is easy to know
that
$$
\mathrm{dim}_{\bbc}\mathrm{Ext}^1_{\bbc Q}(E_1[n],
P_e)=\mathrm{dim}_{\bbc}\mathrm{Hom}_{\bbc Q}(P_e,
E_n[n])=\mathrm{dim}_{\bbc}\mathrm{Hom}_{\bbc Q}(P_e, E_1[n])=1.
$$
The corresponding short exact sequences are
$$
0\rightarrow P_e\rightarrow L\rightarrow E_1[n]\rightarrow 0,
$$
$$
0\rightarrow P_e\rightarrow E_n[n]\rightarrow E_1[n-1]\oplus
E_n/P_e\rightarrow 0
$$
and
$$\quad 0\rightarrow P_e\rightarrow E_1[n]\rightarrow I\rightarrow
0.
$$
The last exact sequence is induced by the following commutative
diagram
$$
\xymatrix{0\ar[r]&P_e\ar@{=}[d]\ar[r]&E_n\ar[d]\ar[r]&E_n/P_e\ar[d]\ar[r]&0\\
0\ar[r]&P_e\ar[r]&E_n[n]\ar[r]&X\ar[r]&0. }
$$
Indeed, since $E_n/P_e$ is a indecomposable preinjective module of
defect $1$ and using the snake lemma on the diagram, we have
$X\cong E_1[n-1]\oplus E_n/P_e.$ Hence, by Theorem \ref{CK}, we
have
$$
X_{P_e}X_{E_1[n]}=X_{L}+X_{E_2[n-1]}X_{\tau^{-1}(E_n/P_e)}.
$$
Consider the commutative diagram of exact sequences
$$
\xymatrix{&&P_e\ar@{=}[r]\ar[d]&P_e\ar[d]&\\0\ar[r]&E_1[n-1]\ar@{=}[d]\ar[r]&E_1[n]\ar[d]\ar[r]&E_n\ar[d]\ar[r]&0\\
0\ar[r]&E_1[n-1]\ar[r]&I\ar[r]&E_n/P_e\ar[r]&0.}
$$
The short exact sequence at the bottom induces a triangle in
$\mc(Q)$
$$
 E_2[n-1]\rightarrow \tau^{-1}I\rightarrow
\tau^{-1}(E_n/P_e)\rightarrow E_1[n-1].
$$
The short exact sequence
$$
0\rightarrow E_2[n-2]\rightarrow E_2[n-1]\rightarrow
E_n\rightarrow 0
$$
induces the short exact sequence
$$
0\rightarrow E_2[n-2]\oplus P_e\rightarrow E_2[n-1]\rightarrow
E_n/P_e\rightarrow 0.
$$
Hence, we have the triangle in $\mc(Q)$
$$
\tau^{-1}(E_n/P_e)\rightarrow E_2[n-2]\oplus P_e\rightarrow
E_2[n-1]\rightarrow E_n/P_e.
$$
 On the other hand, apply
$\mathrm{Hom}(-, E_1[n-1])$ to the short exact sequence
$$
0\rightarrow P_e\rightarrow E_n\rightarrow E_n/P_e\rightarrow 0
$$
to obtain $\mathrm{dim}_{\bbc}\mathrm{Ext}^1_{\bbc Q}(E_n/P_e,
E_1[n-1])=\mathrm{dim}_{\bbc}\mathrm{Ext}^1_{\bbc Q}(E_n,
E_1[n-1])=1.$ Since $\mathrm{Hom}_{\bbc Q}(E_n/P_e, E_2[n-1])=0$,
we deduce
$$
\mathrm{Ext}^{1}_{\mc(Q)}(\tau^{-1}(E_n/P_e), E_2[n-1])=
\mathrm{Ext}^{1}_{\mc(Q)}(E_n/P_e, E_1[n-1])=1.
$$
Hence, using Theorem \ref{CK} again, we have
$$
X_{E_2[n-1]}X_{\tau^{-1}(E_n/P_e)}=X_{\tau^{-1}I}+X_{E_2[n-2]}X_{P_e}.
$$
Let $E$ be a regular simple module of dimension vector $\delta$.
Then by the similar discussion as above and Theorem \ref{CK}, we
have
$$
X_{P_e}X_{E}=X_{L}+X_{\tau^{-1}I}.
$$
Therefore, we obtain the identity
$$
X_{E_1[n]}=X_{E}+X_{E_2[n-2]}.
$$
By Lemma 6.4 in \cite{Dupont1}, we have
$$
X_{E_i[n]}=X_{E}+X_{\tau^{-1}(E_{i}[n-2])}.
$$
for $i=1,\cdots, n$. This identity is called the difference
property.

If $e$ is a source, then $I_e$ is a simple injective module and $(P,
I_e)$ is an orthogonal exceptional pair. Let $\mathcal{T}$ be as
above. Then $\mathrm{reg.socle}(E_n[n])_e=(E_n)_e\neq 0$ and then
$E_n[n]\in \Omega'.$ It is easy to know that
$$
\mathrm{dim}_{\bbc}\mathrm{Ext}^1_{\bbc Q}( I_e,
E_n[n])=\mathrm{dim}_{\bbc}\mathrm{Hom}_{\bbc Q}(E_1[n],
I_e)=\mathrm{dim}_{\bbc}\mathrm{Hom}_{\bbc Q}(E_n[n], I_e)=1.
$$
The corresponding short exact sequences are
$$
0\rightarrow E_n[n]\rightarrow L'\rightarrow I_e\rightarrow 0, \quad
0\rightarrow P\rightarrow E_n[n]\rightarrow I_e\rightarrow 0
$$
and
$$
0\rightarrow Y\rightarrow E_1[n]\rightarrow I_e\rightarrow 0.
$$
Let $$ 0\rightarrow P'\rightarrow E_1\rightarrow I_e\rightarrow 0.
$$ be a short exact sequence. Here, $P'$ is a preprojective module.
Then we have $Y\cong P'\oplus E_1[n-1]$ by the following commutative
diagram
$$
\xymatrix{0\ar[r]&Y\ar[d]\ar[r]&E_1[n]\ar[d]\ar[r]&I_e\ar@{=}[d]\ar[r]&0\\
0\ar[r]&P'\ar[r]&E_n\ar[r]&I_e\ar[r]&0. }
$$ The short exact sequence
$$
0\rightarrow Y\rightarrow E_1[n]\rightarrow I_e\rightarrow 0
$$ induces a triangle in $\mc (Q)$
$$
 I_e\rightarrow \tau P'\oplus E_n[n-1]\rightarrow E_n[n]\rightarrow
 \tau I_e.
$$ Hence, by Theorem \ref{CK}, we
have
$$
X_{I_e}X_{E_n[n]}=X_{L'}+X_{E_n[n-1]}X_{\tau P'}.
$$
The short exact sequence
$$
0\rightarrow P'\rightarrow P\rightarrow E_n[n-1]\rightarrow 0
$$
induces the triangle in $\mc(Q)$
$$
\tau P'\rightarrow \tau P\rightarrow E_n[n-1]\rightarrow
 \tau^2 P'.
$$
The short exact sequence
$$
0\rightarrow P'\rightarrow E_n[n-1]\rightarrow E_1[n-2]\oplus
I_e\rightarrow 0
$$
induces the triangle in $\mc(Q)$
$$
E_n[n-1]\rightarrow E_1[n-2]\oplus I_e\rightarrow \tau P'\rightarrow
E_1[n-1].
$$
Hence, we have
$$
X_{E_n[n-1]}X_{\tau P'}=X_{\tau P}+X_{E_1[n-2]}X_{I_e}.
$$
Let $E$ be a regular simple module of dimension vector $\delta$.
Then by the similar discussion as above and Theorem \ref{CK}, we
have
$$
X_{I_e}X_{E}=X_{L'}+X_{\tau P}.
$$
Therefore, we obtain the identity
$$
X_{E_n[n]}=X_{E}+X_{E_1[n-2]}.
$$
By Lemma 6.4 in \cite{Dupont1}, we have
$$
X_{E_i[n]}=X_{E}+X_{\tau^{-1}(E_{i}[n-2])}.
$$
for $i=1,\cdots, n$. We have proved the following theorem
\begin{Thm}\label{difference}
Let $Q$ be a tame quiver without oriented cycles and $\mathcal{T}$
be a tube of rank $n>1$ with regular simple modules $E_1, \cdots,
E_n$ such that $E_1=\tau E_{2}, \cdots, E_n=\tau E_1$. Assume
$\delta$ is the minimal imaginary root for $Q$. Then we have
$$
X_{E_i[n]}=X_{E}+X_{\tau^{-1}E_{i}[n-2]}
$$
for $i=1, \cdots, n$ and any regular simple module $E$ of dimension
vector $\delta.$
\end{Thm}

In \cite{Dupont1}, the author defined the generic variables for
any acyclic quiver $Q$. In particular, if $Q$ is a tame quiver,
the author showed the the set of generic variables is the
following set
$$
\mathcal{B}_{g}(Q):=\{X_{L}, X^m_{\delta}X_{R}\mid L\in \mc(Q),
\mathrm{Ext}^1(L,L)=0, m\geq 1, R \mbox{ is a regular rigid
module}\}.
$$
The author proved that if the difference property holds for $Q$,
then $\mathcal{B}_{g}(Q)$ is a $\mathbb{Z}$-basis of the cluster
algebra $\mathcal{EH}(Q)$. If $Q$ is of type $\widetilde{A}_{p,
q}$, the author showed the difference property holds. Hence, by
Theorem \ref{difference}, we have the following corollary.
\begin{Cor}\label{unipotent}
The set $\mathcal{B}_{g}(Q)$ is a $\mathbb{Z}$-basis of the
cluster algebra $\mathcal{EH}(Q)$. There is a unipotent transition
matrix between $\mathcal{B}_{g}(Q)$ and $\mathcal{B}(Q).$
\end{Cor}
\begin{proof}\
With the notation in Theorem \ref{difference}, we have
$$
X^m_{\delta}X_{R}=X_{E}^mX_{R}=(X_{E_i[n]-X_{E_{i+1}[n-2]}})^mX_{R}=(X_{E_i[mn]}+\sum_{\mathrm{\underline{dim}}L\prec
m\delta}n_LX_L)X_R
$$
where $L$ is a regular module and $n_L\in \bbz$. Using Lemma
\ref{induction}, Theorem \ref{16} and Theorem \ref{30}, we obtain
$$
(X_{E_i[mn]}+\sum_{\mathrm{\underline{dim}}L\prec
m\delta}n_LX_L)X_R=X_{T\oplus
R'}+\sum_{\mathrm{\underline{dim}}Y\prec
\mathrm{\underline{dim}}(T\oplus R')}n_YX_{Y}.
$$
where both $T\oplus R'$ and $Y$ are  regular modules such that
$X_{T\oplus R'}, X_{Y}\in \mathcal{B}(Q)$ and $n_Y\in \bbz$.  This
completes the proof of the corollary.
\end{proof}

\begin{Cor}\label{un}
There is a unipotent transition matrix between
$\mathcal{B}_{g}(Q)$ and $\mathcal{B}'(Q).$
\end{Cor}
We will illustrate Theorem \ref{difference} by two examples in the
following.
\end{subsection}
\begin{subsection}{}
Let $Q$ be a quiver of type $\widetilde{A}_{p, q}$ as follows.
$$
\xymatrix@R=0.8pc{&2\ar[r]&\cdots\ar[r]&p\ar[rd]&\\
1\ar[ur]\ar[dr]&&&&p+1\\
&p+q\ar[r]&\cdots\ar[r]&p+2\ar[ur]&}
$$
The difference property has been proved to hold for
$\widetilde{A}_{p, q}$ by the different method in \cite{Dupont1}.
Here we give an alternative proof. There are two non-homogeneous
tubes (denoted by $\mathcal{T}_0, \mathcal{T}_\infty$) in the set
of indecomposable regular modules. The minimal imaginary root of
$Q$ is $(1,1, \cdots, 1).$ Let $\lambda\in \bbc^{*}$ and
$E(\lambda)$ be the regular simple module as follows
$$
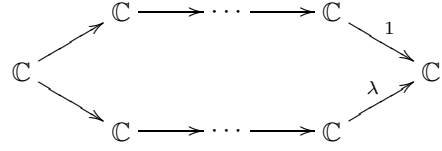
\xymatrix@R=0.8pc{&\bbc\ar[r]&\cdots\ar[r]&\bbc\ar[rd]^{1}&\\
\bbc\ar[ur]\ar[dr]&&&&\bbc\\
&\bbc\ar[r]&\cdots\ar[r]&\bbc\ar[ur]^{\lambda}&}
$$
Its proper submodules are of the forms as follows:
$$
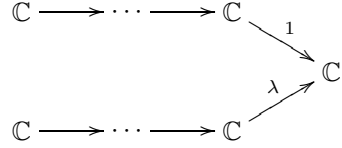
\xymatrix@R=0.8pc{&\bbc\ar[r]&\cdots\ar[r]&\bbc\ar[rd]^{1}&\\
&&&&\bbc\\
&\bbc\ar[r]&\cdots\ar[r]&\bbc\ar[ur]^{\lambda}&}
$$
The regular simple modules in $\mathcal{T}_0$ are
$$ 
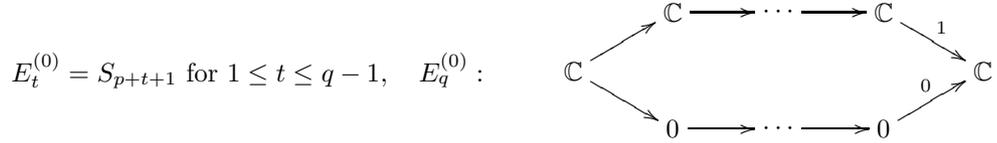
\xymatrix@R=0.5pc{&&\bbc\ar[r]&\cdots\ar[r]&\bbc\ar[rd]^{1}&\\
E^{(0)}_{t}=S_{p+t+1} \mbox{ for }1\leq t\leq q-1,\quad E^{(0)}_q: &\bbc\ar[ur]\ar[dr]&&&&\bbc\\
&&0\ar[r]&\cdots\ar[r]&0\ar[ur]^{0}&}  $$ The regular module
$E^{(0)}_1[q]$ has the form as follows
$$
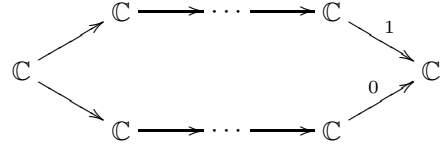
\xymatrix@R=0.8pc{&\bbc\ar[r]&\cdots\ar[r]&\bbc\ar[rd]^{1}&\\
\bbc\ar[ur]\ar[dr]&&&&\bbc\\
&\bbc\ar[r]&\cdots\ar[r]&\bbc\ar[ur]^{0}&}
$$
Its proper submodules are of the following two forms:
$$
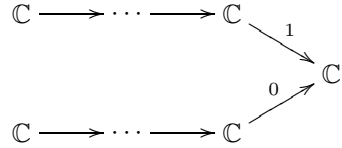
\xymatrix@R=0.8pc{&\bbc\ar[r]&\cdots\ar[r]&\bbc\ar[rd]^{1}&\\
&&&&\bbc\\
&\bbc\ar[r]&\cdots\ar[r]&\bbc\ar[ur]^{0}&}
$$
and
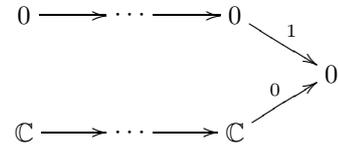
$$
\xymatrix@R=0.8pc{&0\ar[r]&\cdots\ar[r]&0\ar[rd]^{1}&\\
&&&&0\\
&\bbc\ar[r]&\cdots\ar[r]&\bbc\ar[ur]^{0}&}
$$
The proper submodules with the second form are the submodules of
$E^{(0)}_2[q-2]$  with regular socle $S_{p+2}$.  Hence, we have
$$
\chi(Gr_{\ue}(E^{(0)}_1[q]))=\chi(Gr_{\ue}(E(\lambda)))+\chi(Gr_{\ue'}(E^{(0)}_2[q-2]))
$$
for any dimension vector $\ue$ and
$\ue'=\ue-\mathrm{\underline{dim}}S_{p+2}.$ Indeed, by \cite[Lemma
1]{Hubery2005}, we have
$$
\mathrm{\underline{dim}}E^{(0)}_1R+\mathrm{\underline{dim}}E^{(0)}_qR^{tr}=\mathrm{\underline{dim}}E^{(0)}_1+\mathrm{\underline{dim}}E^{(0)}_q.
$$
By definition, we deduce the difference property for
$\widetilde{A}_{p, q}$.
\end{subsection}

\begin{subsection}{}
Let $Q$ be a tame quiver of type $\widetilde{D}_m$ for $m\geq 4$
as follows
$$
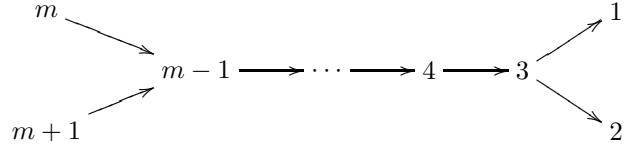
\xymatrix@R=0.8pc{m\ar[rd]&&&&&1\\
&m-1\ar[r]&\cdots\ar[r]&4\ar[r]&3\ar[ru]\ar[rd]&\\
m+1\ar[ru]&&&&&2}
$$
There are three non-homogeneous tubes (denoted by $\mathcal{T}_0,
\mathcal{T}_1, \mathcal{T}_{\infty}$). The minimal imaginary root
of $Q$ is $(1,1,2,\cdots,2,1,1).$ Let $\lambda\in \bbc$ and
$E(\lambda)$ be the regular module as follows:
$$
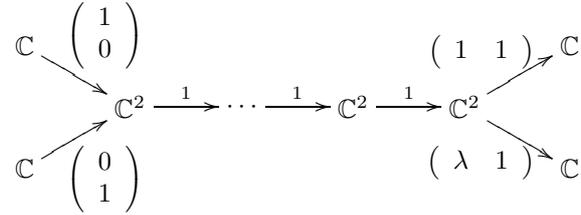
\xymatrix@R=0.8pc{\bbc\ar[rd]^{\left(%
\begin{array}{c}
  1 \\
  0 \\
\end{array}%
\right)}&&&&&\bbc\\
&\bbc^2\ar[r]^1&\cdots\ar[r]^1&\bbc^2\ar[r]^1&\bbc^2\ar[ru]^{\left(%
\begin{array}{cc}
  1 & 1 \\
\end{array}%
\right)}\ar[rd]_{\left(%
\begin{array}{cc}
  \lambda & 1 \\
\end{array}%
\right)}&\\
\bbc\ar[ru]_{\left(%
\begin{array}{c}
  0 \\
  1 \\
\end{array}%
\right)}&&&&&\bbc}
$$
We note that $E(0)\in \mathcal{T}_0$ and $E(1)\in \mathcal{T}_1.$
Replacing $\left(%
\begin{array}{cc}
  \lambda & 1 \\
\end{array}%
\right)$ by $\left(%
\begin{array}{cc}
  1 & 0 \\
\end{array}%
\right)$, we obtain a regular module $E(\infty)\in
\mathcal{T}_{\infty}.$ The regular simple modules in
$\mathcal{T}_1$ are
$$
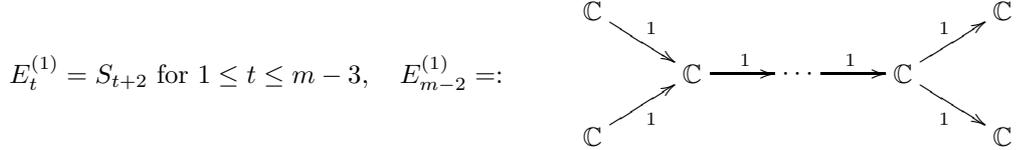
\xymatrix@R=0.6pc{&\bbc\ar[rd]^{1}&&&&\bbc\\
E^{(1)}_t=S_{t+2} \mbox{ for } 1\leq t \leq m-3,
\quad E^{(1)}_{m-2}=:&&\bbc\ar[r]^1&\cdots\ar[r]^1&\bbc\ar[ru]^{1}\ar[rd]_{1}&\\
&\bbc\ar[ru]_{1}&&&&\bbc}
$$
We note that $E(1)=E^{(1)}_1[m-2].$

Let $\ue=(e_1,\cdots, e_{m+1})$ be a dimension vector satisfying
one of the following condition:
\begin{enumerate}
    \item $e_2\neq 0$;
    \item $e_2=0$, $e_3=0$;
    \item $e_2=0$, $e_3=2$.
\end{enumerate}
Let $M$ be a submodule of $E(\lambda)$ for $\lambda\in \bbc^*$
with $\mathrm{\underline{dim}}M=\ue$. Then it is clear that there
are unique submodule $M(0)$ of $E(0)$ and submodule $M(1)$ of
$E(1).$ Hence, we have
$$
\chi(Gr_{\ue}(E(\lambda)))=\chi(Gr_{\ue}(E(0)))=\chi(Gr_{\ue}(E(1)))
\mbox{ and } \chi(Gr_{\ue}(E^{(1)}_1[m-3]))=0
$$
for $\lambda\in \bbc^{*}$. Now we consider the proper submodules
$M$ of $E(1)$ with dimension vector $\ue$ satisfying $e_2=0,
e_3=1.$ If $M=X\oplus Y$. we assume that $X$ satisfies
$(\mathrm{\underline{dim}}X)_3=0$, then $X$ is also the submodule
of $E(\lambda)$ for any $\lambda\in \bbc^*.$ Hence, we can assume
that $M$ is indecomposable with $(\mathrm{\underline{dim}}M)_3=1$
and $(\mathrm{\underline{dim}}M)_2=0$. Then $M$ has the form as
follows
$$
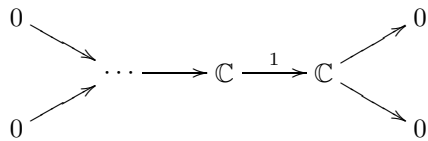
\xymatrix@R=0.7pc{0\ar[rd]&&&&0\\
&\cdots\ar[r]&\bbc\ar[r]^1&\bbc\ar[ru]\ar[rd]&\\
0\ar[ru]&&&&0}
$$
Let $N$ be an indecomposable submodule of $E(\lambda)$ such that
$(\mathrm{\underline{dim}}N)_3=1$ and
$(\mathrm{\underline{dim}}N)_2=0.$ Then $N$ has the form as
follows
$$
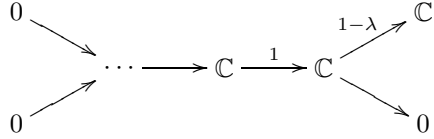
\xymatrix@R=0.7pc{0\ar[rd]&&&&\bbc\\
&\cdots\ar[r]&\bbc\ar[r]^1&\bbc\ar[ru]^{1-\lambda}\ar[rd]&\\
0\ar[ru]&&&&0}
$$
It corresponds to the submodule of $E(1)$
$$
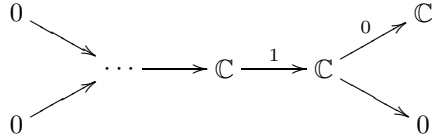
\xymatrix@R=0.7pc{0\ar[rd]&&&&\bbc\\
&\cdots\ar[r]&\bbc\ar[r]^1&\bbc\ar[ru]^{0}\ar[rd]&\\
0\ar[ru]&&&&0}
$$
Note that
$\chi(Gr_{\ue}(E^{(1)}_1[m-3]))=\chi(Gr_{\ue-\mathrm{\underline{dim}}S_3}(E^{(1)}_2[m-4])).$
Hence, we have
$$
\chi(Gr_{\ue}(E(1)))=\chi(Gr_{\ue}(E(\lambda)))+\chi(Gr_{\ue-\mathrm{\underline{dim}}S_3}(E^{(1)}_2[m-4])).
$$

\end{subsection}

\end{document}